\documentclass[12pt]{amsart}
\usepackage{amsmath,amssymb,amscd,amsthm,amsbsy,times}
\usepackage[initials,short-journals]{amsrefs}
\usepackage[shortlabels,inline]{enumitem}
\DeclareFontFamily{U}{rsfs}{\skewchar\font"7F}
\DeclareFontShape{U}{rsfs}{m}{n}{
	<-6> rsfs5
	<6-8> rsfs7
	<8-> rsfs10
	}{}
\DeclareMathAlphabet{\mathscr}{U}{rsfs}{m}{n}
\usepackage[%
 pdftitle={},%
 pdfsubject={},%
 pdfauthor={},%
 pdfkeywords={}%
]{hyperref}
\allowdisplaybreaks
\newcommand{\C}{{\mathbb C}}
\newcommand{\R}{{\mathbb R}}
\newcommand{\Z}{{\mathbb Z}}

\newcommand{\gl}{\mathfrak{gl}}
\newcommand{\GL}{\mathrm{GL}}

\newcommand{\PGL}{\mathrm{PGL}}

\DeclareMathOperator{\Ad}{\mathrm{Ad}}

\newcommand{\pdif}[2]{\dfrac{\partial#1}{\partial#2}}
\newcommand{\norm}[1]{\lvert#1\rvert}

 \newtheorem{theorem}{Theorem}[section]
 \newtheorem{corollary}[theorem]{Corollary}
 \newtheorem{proposition}[theorem]{Proposition}
 \newtheorem{lemma}[theorem]{Lemma}

 \newtheorem{remark}[theorem]{Remark}
\theoremstyle{definition}
 \newtheorem{definition}[theorem]{Definition}
 \newtheorem{notation}[theorem]{Notation}
 \newtheorem{example}[theorem]{Example}
\title{Notes on projective structures with torsion}
\author{Taro Asuke}
\address{Graduate School of Mathematical Sciences, University of Tokyo, 3-8-1 Komaba, Meguro-ku, Tokyo 153-8914, Japan}
\email{asuke@ms.u-tokyo.ac.jp}
\date\today
\thanks{The auther is partially supported by JSPS KAKENHI Grant Number JP21H00980.}
\begin{document}
\baselineskip=16pt
\begin{abstract}
We show that projective structures with torsion are described in terms of affine connections in a parallel way as in the torsion-free case which is done by Kobayashi and Nagano.
For this, we make use of a bundle of formal frames, which is a generalization of a bundle of frames.
We will also describe projective structures in terms of Thomas--Whitehead connections by following Roberts.
In particular, we introduce normal projective connections and show the fundamental theorem for Thomas--Whitehead connections regardless the triviality of the torsion.
We will study some examples of projective structures of which the torsion is non-trivial while the curvature is trivial.
In this article, projective structures are considered to be the same if they have the same geodesics ignoring parameters and the same torsions.
\end{abstract}

\maketitle

\section*{Introduction}
Projective structures are quite well-studied.
They can be described by Cartan connections and frame bundles, as studied by Kobayashi and Nagano~\cite{Kobayashi-Nagano}, et.~al.
Projective structures can be also described in terms of Thomas--Whitehead connections (TW-connections for short) which are linear connections on a certain line bundle~\cite{Roberts}.
Associated with projective structures are torsions, which are $2$-forms.
Such structures appears in several contexts, for example in the study of foliations in our point of view~\cite{asuke:2015},~\cite{asuke:2017}.
If the torsion of a projective structure vanishes, then the structure is said to be \textit{torsion-free} or without torsion.
Actually, the above-mentioned studies are done in the torsion-free case.
One of the most fundamental results is the existence of normal projective connections~\cite{Kobayashi-Nagano}*{Proposition~3} which is a Cartan connection of special kind.
A corresponding result for TW-connections is known as the Fundamental theorem for TW-connections~\cite{Roberts}.
On the other hand, linear connections always induce projective structures even if they are with torsions.
Such geometric structures with torsions are previously studied by Morimoto~\cite{Morimoto1983},~\cite{Morimoto1993}, McKay~\cite{McKay}, et.~al in terms of Cartan connections.
In this article, we will show that we can modify Kobayashi's formulation to apply so that we can treat projective structures with torsions using affine connections in a parallel way as in~\cite{Kobayashi-Nagano}.
For this purpose, we need a notion of formal frame bundles~\cite{asuke:2022} which is a generalization of frame bundles.
Usually, a $2$-frame at a point is given by a pair $(a^i{}_j,a^i_{jk})\in\GL_n(\R)\times\R^{n^3}$ such that $a^i{}_{jk}=a^i{}_{kj}$.
The symmetricity condition is quite related with torsion-freeness and we have to drop this condition in order to deal with torsions.
This leads us to formal frames.
A formal $2$-frame at a point is a pair $(a^i{}_j,a^i_{jk})\in\GL_n(\R)\times\R^{n^3}$.
We refer to~\cite{asuke:2022} for the precise definition and details of formal frames.
We also introduce a version of TW-connections applicable for projective structures with torsions.
Finally, expecting a better understanding of the torsion, we will study some examples of projective structures of which the torsion is non-trivial while the curvature is trivial.
\par
In this article, projective structures are considered to be the same if they have the same (unparameterized) geodesics and the same torsions except last part of Section~2.
Throughout this article, $(U,\varphi)$ and $(\widehat{U},\widehat{\varphi})$ denote charts, and $\psi$ denotes the transition function.
Representing (local) tensors, we make use of the Einstein convention.
For example, $a^i{}_\alpha b^\alpha{}_{jk}$ means $\sum_{\alpha}a^i{}_\alpha b^\alpha_{jk}$.
The range of $\alpha$ will be from $1$ to $\dim M$ or from $1$ to $\dim M+1$.
We basically retain notations of \cite{Kobayashi-Nagano} and \cite{Roberts}.
Finally, the order of lower indices of the Christoffel symbols are reversed in this article (see Notation~\ref{not2.12}).

\section{Cartan connections}
We recall basics of Cartan connections after~\cite{K}.
We will work in the real category, however, we can work in the complex category (not necessarily the holomorphic category)  after obvious modifications.

Let $G$ be a Lie group and $H$ a closed subgroup of $G$.
We assume that $P$ is a principal $H$-bundle over $M$.
In what follows, the Lie algebra is represented by the corresponding lower German letter, e.g., $\mathfrak{g}$ will denote the Lie algebra of $G$.

\begin{definition}
A \textit{Cartan connection} is a $1$-form $\omega$ on $P$ with values in $\mathfrak{g}$ which satisfies the following conditions:
\begin{enumerate}
\item
$\omega(A^*)=A$ for any $A\in\mathfrak{h}$, where $A^*$ denotes the fundamental vector field associated with $A$.
\item
$R_a{}^*\omega=\Ad_{a^{-1}}\omega$ for any $a\in H$.
\item
$\omega(X)\neq0$ for any non-zero vector $X$ on $P$.
\end{enumerate}
\end{definition}

\begin{notation}
In what follows, we assume that $G=\PGL_{n+1}(\R)=\GL_{n+1}(\R)/Z$, where $Z=\{\lambda I_{n+1}\mid\lambda\neq0\}$.
Let $[x^0:\cdots:x^n]$ be the homogeneous coordinates for $\R P^n$, and $H\subset G$ the isotopy group of $[0:\cdots:0:1]$.
Finally we set $\mathfrak{m}=\R^n$, which is understood as a space of column vectors, and let $\mathfrak{m}^*$ denote its dual.
\end{notation}

\begin{definition}
We set
\begin{align*}
G_0&=\left\{\begin{pmatrix}
A & 0\\
\xi & a
\end{pmatrix}\in\GL_{n+1}(\mathbb{R})\;\middle|\;a\det A=1\right\}\hskip-12pt\left.\phantom{\biggl(}\middle/\right.\hskip-3pt Z,\\*
G_1&=\left\{\begin{pmatrix}
I_n & 0\\
\xi & 1
\end{pmatrix}\in\GL_{n+1}(\mathbb{R})\;\middle|\;\xi\in\mathfrak{m}^*\right\}.
\end{align*}
\end{definition}
Note that $G_1$ is naturally a subgroup of $G$ and $G_0$.
We have
\begin{align*}
\mathfrak{g}_0&=\left\{\begin{pmatrix}
A & 0\\
0 & a
\end{pmatrix}\;\middle|\;\mathop{\mathrm{tr}}A+a=0\right\},\\*
\mathfrak{g}_1&=\left\{\begin{pmatrix}
0 & 0\\
\xi & 0
\end{pmatrix}\right\}.
\end{align*}
If we set
\[
\mathfrak{g}_{-1}=\left\{\begin{pmatrix}
0 & v\\
0 & 0
\end{pmatrix}\right\},
\]
then we have
\[
\mathfrak{g}=\mathfrak{g}_{-1}\oplus\mathfrak{g}_0\oplus\mathfrak{g}_1.
\]
We have $\mathfrak{g}_{-1}\cong\mathfrak{m}$, $\mathfrak{g}_0\cong\mathfrak{gl}_n(\R)$ and $\mathfrak{g}_1\cong\mathfrak{m}^*$ so that $\mathfrak{g}\cong\mathfrak{m}\oplus\gl_n(\R)\oplus\mathfrak{m}^*$.
We also have $\mathfrak{h}\cong\gl_n(\R)\oplus\mathfrak{m}^*$.
The identifications are given by
\begin{alignat*}{3}
&\begin{pmatrix}
0 & v\\
0 & 0
\end{pmatrix}\in\mathfrak{g}_{-1}&{}\mapsto{}& v\in\mathfrak{m},\\*
&\begin{pmatrix}
A & 0\\
0 & a
\end{pmatrix}\in\mathfrak{g}_0&{}\mapsto{}& U=A-aI_n\in\mathfrak{gl}_n(\R),\\*
&\begin{pmatrix}
0 & 0\\
\xi & 0
\end{pmatrix}\in\mathfrak{g}_1&{}\mapsto{}& \xi\in\mathfrak{m}^*.
\end{alignat*}
Note that $U\in\mathfrak{gl}_n(\R)$ corresponds to $\begin{pmatrix}
U & 0\\
0 & 0
\end{pmatrix}-\dfrac1{n+1}(\mathop{\mathrm{tr}}U)I_{n+1}$.
Under these identifications, the Lie brackets are given as follows.
Let $u,v\in\mathfrak{m}$, $u^*,v^*\in\mathfrak{m}^*$ and $U,V\in\gl_n(\R)$.
Then, we have
\begin{align*}
&[u,v]=0,\\*
&[u^*,v^*]=0,\\*
&[U,u]=Uu\in\mathfrak{m},\\*
&[u^*,U]=u^*U\in\mathfrak{m}^*,\\*
&[U,V]=UV-VU\in\gl_n(\R),\\*
&[u,u^*]=uu^*+u^*uI_n\in\gl_n(\R).
\end{align*}
In what follows, we always make use of these identifications.
If $\omega$ is a Cartan connection on $P$, then we represent $\omega=(\omega^i,\omega^i{}_j,\omega_j)$ according to the identification $\mathfrak{g}=\mathfrak{m}\oplus\gl_n(\R)\oplus\mathfrak{m}^*$.

\begin{remark}
\label{rem4.3}
Each element $g$ of\/ $\mathrm{PGL}_{n+1}(\R)$ admits a representative of the form $\begin{pmatrix}
A & \xi^*\\
\xi & 1
\end{pmatrix}$.
By associating $g$ with $(\xi,A,\xi^*)$, we can consider $(a^i,a^i{}_j,a_j)$ as coordinates for $\mathrm{PGL}_{n+1}(\R)$.
With respect to these coordinates, we have $H=\{(0,a^i{}_j,a_j)\}$.
Let $o=[0:\cdots:0:1]$ denote $H\in\PGL_{n+1}(\R)/H$.
If $h=(0,a^i{}_j,a_j)\in H$ and if $x=(x^i)=[x^1:\cdots:x^n:1]$ is close enough to $o$, then we have
\begin{align*}
h.x&=\frac{a^i{}_jx^j}{a_jx^j+1}\\*
&=a^i{}_jx^j-a^i{}_jx^ja_kx^k+\cdots\\*
&=a^i{}_jx^j-\frac12(a^i{}_ja_k+a^i{}_ka_j)x^jx^k+\cdots.
\end{align*}
\end{remark}

\begin{definition}
\label{def1.2}
Let $\omega=(\omega^i,\omega^i{}_j,\omega_j)$ be a Cartan connection on $P$.
We set
\begin{alignat*}{3}
&\Omega^i&{}={}&d\omega^i+\omega^i{}_k\wedge\omega^k,\\*
&\Omega^i{}_j&{}={}&d\omega^i{}_j+\omega^i{}_k\wedge\omega^k{}_j+\omega^i\wedge\omega_j-\delta^i{}_j\omega_k\wedge\omega^k,\\*
&\Omega_j&{}={}&d\omega_j+\omega_k\wedge\omega^k{}_j.
\end{alignat*}
We call $\Omega^i$ the \textit{torsion} and $(\Omega^i{}_j,\Omega_j)$ the \textit{curvature} of $\omega$, respectively.
\end{definition}

We refer to $(\Omega^i{}_j)$ as the \textit{curvature matrix} of $\omega$ and consider trace of it.

We have the following

\begin{proposition}[\cite{Kobayashi-Nagano}*{Proposition~2}].
We can represent the torsion and the curvature as
\begin{alignat*}{4}
&\Omega^i&{}={}&\frac12K^i{}_{kl}\omega^j\wedge\omega^k,&\quad &K^i{}_{lk}=-K^i{}_{kl},\\*
&\Omega^i{}_j&{}={}&\frac12K^i{}_{jkl}\omega^k\wedge\omega^l,&\quad &K^i{}_{jlk}=-K^i{}_{jkl},\\*
&\Omega_j&{}={}&\frac12K_{jkl}\omega^k\wedge\omega^l,&\quad &K_{jlk}=-K_{jkl},
\end{alignat*}
where $K^i{}_{kl}$, $K^i{}_{jkl}$ and $K_{jkl}$ are functions on $P$.
\end{proposition}

\begin{remark}
\label{rem4.4}
If $\omega$ is a Cartan connection on $P$, then we have the following\textup{:}
\begin{enumerate}[\textup{\alph{enumi})}]
\item
$\omega^i(A^*)=0$ and $\omega^i{}_j(A^*)=A^i{}_j$ for any $A=(A^i{}_j,A_j)\in\mathfrak{h}=\gl_{n+1}(\R)\oplus\mathfrak{m}^*$.
\item
$R_a{}^*(\omega^i,\omega^i{}_j)=\Ad_{a^{-1}}(\omega^i,\omega^i{}_j)$ for any $a\in H$.
\item
Let $X\in TP$.
We have $\omega^i(X)=0$ if and only if $X$ is vertical, namely, tangent to a fiber of $P\to M$.
\end{enumerate}
\end{remark}

Proposition~3 in~\cite{Kobayashi-Nagano} holds in the following form.
A point is that we do not need the condition $\Omega^i{}_i=0$.
See also Remark~\ref{rem2.4}.

\begin{proposition}
\label{prop4.5}
Let $\omega^i$ and $\omega^i{}_j$ satisfy the conditions in Remark~\ref{rem4.4}.
Then, there is a Cartan connection of the form $\omega=(\omega^i,\omega^i{}_j,\omega_j)$.
If $n\geq2$, there uniquely exists a Cartan connection such that $K^i{}_{jil}=0$, that is, $\omega$ is Ricci-flat.
If moreover $n\geq3$ and if $\omega$ is torsion-free, then $\Omega^i{}_i=0$, namely, the curvature matrix $(\Omega^i{}_j)$ is trace-free.
\end{proposition}
\begin{proof}
First we show the existence of a Cartan connection.
Let $\{U_\alpha\}$ be a locally finite open covering of $M$ and $\{f_\alpha\}$ a partition of unity subordinate to $\{U_\alpha\}$.
Let $\pi\colon P\to M$ is the projection.
Suppose that for each $\alpha$, there is a Cartan connection $\omega_\alpha$ on $\pi^{-1}(U_\alpha)$ such that $\omega_\alpha=(\omega^i,\omega^i{}_j,\omega_{j,\alpha})$ for some $\omega^i{}_{j,\alpha}$.
If we set $\omega=\sum(f_\alpha\circ\pi)\omega_\alpha$, then $\omega$ is a Cartan connection of the form $(\omega^i,\omega^i{}_j,\omega_j)$.
On the other hand, we may assume that $\pi^{-1}(U_\alpha)$ is trivial.
We fix a trivialization $\pi^{-1}(U_\alpha)\cong U_\alpha\times H$.
If $(x,h)\in U_\alpha\times H$ and if $Y\in T_{(x,h)}P$, then we can represent $Y$ as $Y=X+A$, where $Y\in T_xM$ and $A\in\mathfrak{h}$.
If we set $\omega_\alpha(Y)=\Ad_{a^{-1}}(\omega^i(X),\omega^i{}_j(X),0)+A$, then $\omega_\alpha$ is a Cartan connection of the form $(\omega^i,\omega^i{}_j,\omega_{j\alpha})$.\par
From now on, we assume that $n\geq2$.
We show the uniqueness.
Suppose that $\omega=(\omega^i,\omega^i{}_j,\omega_j)$ and $\omega'=(\omega^i,\omega^i{}_j,\omega'_j)$ are Cartan connections as in the proposition.
By the conditions a) and c), we have $\omega_j-\omega'_j=A_{jk}\omega^k$ for some functions $A_{jk}$ on $P$.
We have
\[
\Omega^i{}_j-\Omega'{}^i{}_j=\omega^i{}\wedge(\omega_j-\omega'_j)-\delta^i{}_j(\omega_k-\omega'_k)\wedge\omega^k.
\]
It follows that
\[
K^i{}_{jkl}-K'{}^i{}_{jkl}=-\delta^i{}_lA_{jk}+\delta^i{}_kA_{jl}+\delta^i{}_jA_{kl}-\delta^i{}_jA_{lk}.
\]
Therefore, we have
\begin{align*}
K^i{}_{jil}-K'{}^i{}_{jil}&=-\delta^i{}_lA_{ji}+\delta^i{}_iA_{jl}+\delta^i{}_jA_{il}-\delta^i{}_jA_{li}\\*
&=(n-1)A_{jl}+(A_{jl}-A_{lj})\\*
&=nA_{jl}-A_{lj}.
\end{align*}
It follows that
\[
\tag{\thetheorem{}a}
\label{eq4.5-a}
A_{jk}=\frac1{n^2-1}(n(K^i{}_{jik}-K'{}^i{}_{jik})+(K^i{}_{kij}-K'{}^i{}_{kij})).
\]
Since $\omega$ and $\omega'$ are Ricci-flat, we have $A_{jk}=0$.

Next, we show that the existence of a Cartan connection which is Ricci-flat.
Let $\omega'$ be a Cartan connection of the form $(\omega^i,\omega^i{}_j,\omega_j')$ which is not necessarily Ricci-flat.
If $\omega$ is a Cartan connection which is Ricci-flat, then we have by \eqref{eq4.5-a} that
\[
\tag{\thetheorem{}b}
\label{eq4.5-3}
A_{jk}=-\frac1{n^2-1}(nK'{}^i{}_{jik}+K'{}^i{}_{kij}).
\]
If we conversely define $A_{jk}$ by the equality~\eqref{eq4.5-3} and set $\omega_j=\omega'_j+A_{jk}\omega^k$, then $(\omega^i,\omega^i{}_j,\omega_j)$ is a desired Cartan connection.\par
Finally, we assume that $\omega$ is torsion-free.
Then $\Omega^i{}_i=0$ by Proposition~\ref{prop4.7} provided that $\dim M\geq3$.
\end{proof}

\begin{proposition}
\label{prop4.7}
Suppose that $n\geq3$ and let $\omega=(\omega^i,\omega^i{}_j,\omega_j)$ be a Cartan connection.
Then, we have the following\textup{:}
\begin{enumerate}[\textup{\theenumi)}]
\item
If $d\Omega^i+\omega^i{}_j\wedge\Omega^j=0$, then we have $K^i{}_{jkl}+K^i{}_{klj}+K^i{}_{ljk}=0$.
\item
If $d\Omega^i+\omega^i{}_j\wedge\Omega^j=0$ and if $K^i{}_{jil}=0$, then $\Omega^i{}_i=0$.
\item
If\/ $\Omega^i=0$ and if\/ $\Omega^i{}_i=0$, then we have $K_{jkl}+K_{klj}+K_{ljk}=0$.
\item
If\/ $\Omega^i=0$ and if\/ $\Omega^i{}_j=0$, then $\Omega_j=0$.
\end{enumerate}
\end{proposition}
\begin{proof}
First we will show 1).
We have $\Omega^i=d\omega^i+\omega^i{}_j\wedge\omega^j$.
Hence we have
\begin{align*}
&\hphantom{{}={}}%
d\Omega^i+\omega^i{}_j\wedge\Omega^j\\*
&=d\omega^i{}_j\wedge\omega^j-\omega^i{}_j\wedge d\omega^j+\omega^i{}_j\wedge(d\omega^j+\omega^j{}_k\wedge\omega^k)\\*
&=d\omega^i{}_j\wedge\omega^j+\omega^i{}_j\wedge\omega^j{}_k\wedge\omega^k\\*
&=\Omega^i{}_j\wedge\omega^j\\*
&=\frac12K^i{}_{jkl}\omega^k\wedge\omega^l\wedge\omega^j.
\end{align*}
It follows that $K^i{}_{jkl}+K^i{}_{klj}+K^i{}_{ljk}=0$ if $d\Omega^i+\omega^i{}_j\wedge\Omega^j=0$.
Next, we show 2).
Suppose in addition that $K^i{}_{jil}=0$.
Then, we have $0=K^i{}_{ikl}+K^i{}_{kli}=K^i{}_{ikl}-K^i{}_{kil}=K^i{}_{ikl}$.
Next, we show 3).
We have
\begin{align*}
d\Omega^i{}_j%
&=d\omega^i{}_k\wedge\omega^k{}_j-\omega^i{}_k\wedge d\omega^k{}_j+d\omega^i\wedge\omega_j-\omega^i\wedge d\omega_j-\delta^i{}_j(d\omega_k\wedge\omega^k-\omega_k\wedge d\omega^k)\\*
&=(\Omega^i{}_k-\omega^i{}_l\wedge\omega^l{}_k-\omega^i\wedge\omega_k+\delta^i{}_k\omega_l\wedge\omega^l)\wedge\omega^k{}_j\\*
&\hphantom{{}={}}%
-\omega^i{}_k\wedge(\Omega^k{}_j-\omega^k{}_l\wedge\omega^l{}_j-\omega^k\wedge\omega_j+\delta^k{}_j\omega_l\wedge\omega^l)\\*
&\hphantom{{}={}}%
+(\Omega^i-\omega^i{}_k\wedge\omega^k)\wedge\omega_j-\omega^i\wedge(\Omega_j-\omega_k\wedge\omega^k{}_j)\\*
&\hphantom{{}={}}%
-\delta^i{}_j((\Omega_k\wedge\omega^k-\omega_l\wedge\omega^l{}_k)\wedge\omega^k-\omega_k\wedge(\Omega^k+\omega^k{}_l\wedge\omega^l))\\*
&=\Omega^i{}_k\wedge\omega^k{}_j-\omega^i{}_k\wedge\Omega^k{}_j+\Omega^i\wedge\omega_j-\omega^i\wedge\Omega_j-\delta^i{}_j(\Omega_k\wedge\omega^k-\omega_k\wedge\Omega^k).
\end{align*}
Taking the trace, we obtain
\[
d\Omega^i{}_i=(n+1)(\Omega^i\wedge\omega_i-\omega^i\wedge\Omega_i).
\]
If\/ $\Omega^i=0$ and if\/ $\Omega^i{}_i=0$, then we have $\omega^i\wedge\Omega_i=0$.
Hence $K_{jkl}+K_{klj}+K_{ljk}=0$.
Finally, we show 4).
If $\Omega^i=0$ and if $\Omega^i{}_j=0$, then we have $\omega^i{}\wedge\Omega_j=0$ by 3).
As $n\geq3$, we have $\Omega_i=0$.
\end{proof}

\section{Cartan connections, affine connections and projective structures}
We follow the arguments in~\cite{Kobayashi-Nagano}, taking torsions into account.

First, we briefly recall bundles of formal frames $\widetilde{P}^r(M)$ and groups $\widetilde{G}^r$ which act on $\widetilde{P}^r(M)$ on the right~\cite{asuke:2022}, where $r=1,2$.

Let $M$ be a manifold, and $P^r(M)$ and $G^r$ the bundle of $r$-frames and the group of $r$-jets~\cite{K}.

\begin{definition}
\begin{enumerate}
\item
We set $\widetilde{P}^1(M)=P^1(M)$ and $\widetilde{G}^1=G^1\cong\GL_n(\R)$.
\item
We set $\widetilde{G}^2=\GL_n(\R)\ltimes\R^{n^3}$, where the multiplication law is given by $(a^i{}_j,a^i{}_{jk})(b^i{}_j,b^i{}_{jk})=(a^i{}_lb^l{}_j,a^i{}_lb^l{}_{jk}+a^i{}_{lm}b^l{}_jb^m{}_k)$ which is the same as the one in $G^2$.
Indeed, $G^2=\{(a^i{}_j,a^i{}_{jk})\in\widetilde{G}^2\mid a^i{}_{jk}=a^i{}_{kj}\}$.
\end{enumerate}
\end{definition}

The group $\widetilde{G}^2$ consists of the $1$-jets of certain bundle homomorphisms, and the bundle $\widetilde{P}^2(M)$ is a principal $\widetilde{G}^2$-bundle which also consists of the $1$-jets of certain bundle homomorphisms.
We have $\widetilde{P}^2(M)=P^2(M)\times_{G^2}\widetilde{G}^2$.

In view of Remark~\ref{rem4.3}, we introduce the following
\begin{definition}
\label{def2.1}
We define a subgroup $H^2$ of $\widetilde{G}^2$ by setting
\[
H^2=\{(a^i{}_j,a^i{}_{jk})\in\widetilde{G}^2\mid\exists\,a_i,\ a^i{}_{jk}=-(a^i{}_ja_k+a_ja^i{}_k)\}.
\]
We regard $(a^i{}_j,a_j)$ as coordinates for $H^2$.
\end{definition}

It is easy to see that $H^2$ is indeed a subgroup of $\widetilde{G}^2$ isomorphic to $H$ and satisfies $G^1=\widetilde{G}^1<H^2<G^2<\widetilde{G}^2$.

\begin{definition}
\begin{enumerate}
\item
A \textit{projective structure} on $M$ is a subbundle $P$ of $\widetilde{P}^2(M)$ with structure group $H^2$.
\item
A \textit{projective connection} associated with a projective structure $P$ is a Cartan connection $\omega=(\omega^i,\omega^i{}_j,\omega_j)$ on $P$ such that $\omega^i$ coincides with the restriction of the canonical form of order $0$ to $P$.
In order to distinguish from TW-connections, we refer to projective connections also as \textit{Cartan projective connections}.
\end{enumerate}
\end{definition}

\begin{remark}
Let $(\theta^i,\theta^i{}_j)$ be the canonical form on $\widetilde{P}^2(M)$.
We set $\Theta^i=d\theta^i+\theta^i{}_j\wedge\theta^j$.
Then we have $\sigma^*\Omega^i=\sigma^*\Theta^i$.
We have $\Theta^i=0$ on $P^2(M)$.
Indeed, this is just the structural equation.
See~\cite{asuke:2022}\, for details.
\end{remark}

\begin{theorem}[cf.~\cite{McKay}*{Theorem~7}]
\label{thm5.3}
For each projective structure $P$ of a manifold $M$, there is a projective connection $\omega=(\omega^i,\omega^i{}_j,\omega_j)$ with the projective structure $P$.
If $n\geq2$, then there exists a unique $\omega$ with the following properties\textup{:}
\begin{enumerate}[\textup{\theenumi)}]
\item
$(\omega^i,\omega^i{}_j)$ coincides with the restriction of the canonical form on $\widetilde{P}^2(M)$ to~$P$.
\item
$K^i{}_{jil}=0$.
\end{enumerate}
If moreover $\omega$ is torsion-free, namely, if $\Omega^i=0$, then $\Omega^i{}_i=0$, or equivalently, $K^i{}_{ikl}=0$.
\end{theorem}
\begin{proof}
This is a consequence of Proposition~\ref{prop4.5}.
Indeed, the restriction of the canonical form satisfies the conditions in Remark~\ref{rem4.4}.
If $n=2$, then the last part will be later shown as Lemma~\ref{lem2.13}.
\end{proof}

\begin{remark}
\label{rem2.4}
Theorem~\ref{thm5.3} is well-known in the torsion-free case.
Since we do not assume projective structures to be torsion-free, we need canonical forms on $\widetilde{P}^2(M)$ which realize torsions.
A point is that the condition $\Omega^i{}_i=0$ is not needed for the uniqueness in Proposition~\ref{prop4.5}.
\end{remark}

\begin{remark}
\label{rem2.5}
Let $(U,\varphi)$ be a chart.
Then, $u\in\widetilde{P}^2(M)|_U$ naturally corresponds to $(u^i,u^i{}_j,u^i{}_{jk})\in\R^n\times\GL_n(\R)\times\R^{n^3}$, which are called the natural coordinates \textup{(}\cite{Kobayashi-Nagano}*{p.~225}, \cite{asuke:2022}*{Definition~1.8}\textup{)}.
If $u\in P^2(M)$ and if $u$ is represented by $f\colon\R^n\to M$, then $(u^i,u^i{}_j,u^i{}_{jk})=\left(f^i(o),DF^i{}_j(o),D^2F^i{}_{jk}(o)\right)$.
The canonical form $(\theta^0,\theta^1)$ is represented as
\begin{align*}
\theta^0{}_u&=v^i{}_\alpha du^\alpha,\\*
\theta^1{}_u&=v^i{}_\alpha du^\alpha{}_j-v^i{}_\alpha u^\alpha{}_{j\beta}v^\beta{}_\gamma du^\gamma,
\end{align*}
where $(v^i{}_j)=(u^i{}_j)^{-1}$.
\end{remark}

\begin{definition}
Let $n\geq2$.
The projective connection given by Theorem~\ref{thm5.3} is called the \textit{normal projective connection} associated with $P$.
\end{definition}

The following is clear.

\begin{proposition}
\begin{enumerate}[\textup{\theenumi)}]
\item
There is a one-to-one correspondence between the following objects\textup{:}
\begin{enumerate}[\textup{\alph{enumii})}]
\item
Sections from $M$ to $\widetilde{P}^2(M)/\widetilde{G}^1$.
\item
Sections from $\widetilde{P}^1(M)$ to $\widetilde{P}^2(M)$ equivariant under the $\widetilde{G}^1$-action.
\item
Affine connections on $M$.
\end{enumerate}
\item
There is a one-to-one correspondence between the following objects\textup{:}
\begin{enumerate}[\textup{\alph{enumii})}]
\item
Sections from $M$ to $\widetilde{P}^2(M)/H^2$.
\item
Projective structures on $M$.
\end{enumerate}
\end{enumerate}
\end{proposition}

If $\nabla$ is an affine connection, then $\nabla$ corresponds to a section from $M$ to $\widetilde{P}^2(M)/\widetilde{G}^1$.
Since $\widetilde{G}^1=G^1$ is a subgroup of $H^2$, $\nabla$ induces a section from $M$ to $\widetilde{P}^2(M)/H^2$, namely, a projective structure.
Conversely, given a projective structure, we can find an affine connection which induces the projective structure because $H^2/\widetilde{G}^1$ is contractible.

We introduce the following definition after~\cite{K} (see also Tanaka~\cite{Tanaka}, Weyl~\cite{Weyl}).

\begin{definition}
\label{def2.8}
Let $\nabla$ and $\nabla'$ be linear connections on $TM$.
Let $\omega$ and $\omega'$ be the connection forms of associated connection on $P^1(M)$.
We say that $\nabla$ and $\nabla'$ are \textit{projectively equivalent} if there is an $\mathfrak{m}^*$-valued function, say $p$, on $P^1(M)$ such that
\[
\omega'-\omega=[\theta,p],
\]
where $\theta$ denotes the canonical form on $P^1(M)$.
\end{definition}
Note that $p$ necessarily satisfy $R_g{}^*p=pg$, where $g\in\GL_n(\R)$.

\begin{remark}
The torsion is invariant under the projective equivalences in the sense of Definition~\ref{def2.8}.
On the other hand, we can consider the usual equivalence relation based on unparameterized geodesics, then any affine connection is equivalent to a torsion-free one.
See Corollary~\ref{cor2.26} and Remark~\ref{rem2.23}.
\end{remark}

\begin{lemma}
\label{lem1.12}
Linear connections $\nabla$ and $\nabla'$ on $TM$ are projectively equivalent if and only if there is a $1$-form, say $\rho$, on $M$ such that $\nabla'-\nabla=\rho\otimes\mathrm{id}+\mathrm{id}\otimes\rho$.
\end{lemma}
\begin{proof}
If $\nabla$ and $\nabla'$ are projectively equivalent, then there is an $\mathfrak{m}^*$-valued function $p$ such that $\omega'-\omega=[\theta,p]$.
If $x\in M$ and if $v\in T_xM$, then we fix a frame $u$ of $T_x M$ and represent $v=uw$.
We set $\rho_x(v)=p(u)w$, and we have $\nabla'-\nabla=\rho\otimes\mathrm{id}+\mathrm{id}\otimes\rho$.
Conversely if $\nabla'-\nabla=\rho\otimes\mathrm{id}+\mathrm{id}\otimes\rho$ holds for a $1$-form $\rho$.
Let $u=(e_1,\ldots,e_n)$ be a frame and $(e^1,\ldots,e^n)$ its dual.
We represent $\rho$ as $\rho=\rho_1e^1+\cdots+\rho_ne^n$ and set $p(u)=(\rho_1,\ldots,\rho_n)$.
Then we have $\omega'-\omega=[\theta,p]$.
\end{proof}

\begin{remark}
Let $(x^1,\ldots,x^n)$ be local coordinates and choose $\left(\pdif{}{x^1},\ldots,\pdif{}{x^n}\right)$ as a frame.
If we represent $\rho$ as $\rho=\rho_idx^i$, then we have
\begin{align*}
(\rho\otimes\mathrm{id})^i{}_{jk}&=\delta^i{}_j\rho_k,\\*
(\mathrm{id}\otimes\rho)^i{}_{jk}&=\delta^i{}_k\rho_j,
\end{align*}
where $\delta^i{}_j=\begin{cases}
1, & i=j,\\
0, & i\neq j
\end{cases}$.
\end{remark}

\begin{lemma}
\label{lem1.14}
If we have $\nabla'-\nabla=\rho\otimes\mathrm{id}+\mathrm{id}\otimes\rho=\rho'\otimes\mathrm{id}+\mathrm{id}\otimes\rho'$, then $\rho'=\rho$.
\end{lemma}
\begin{proof}
We have $(\rho\otimes\mathrm{id}+\mathrm{id}\otimes\rho)(e_i,e_i)=2\rho(e_i)$.
Hence $\rho(e_i)=0$ if $\rho\otimes\mathrm{id}+\mathrm{id}\otimes\rho=0$.
\end{proof}

We will make use of the Christoffel symbols reversing the order of lower indices.
This is convenient when formal frames are considered.
\begin{notation}
\label{not2.12}
We set $\Gamma^i{}_{jk}=dx^i\left(\nabla_{\textstyle{\frac{\partial}{\partial x^k}}}\pdif{}{x^j}\right)$.
\end{notation}

\begin{lemma}
Affine connections $\nabla$ and $\nabla'$ induce the same projective structure if and only if they are projectively equivalent.
\end{lemma}
\begin{proof}
Let $\Gamma^i{}_{jk}$ and $\Gamma'{}^i{}_{jk}$ be the Christoffel symbols for $\nabla$ and $\nabla'$, respectively.
Then, $\nabla$ corresponds to a section from $M$ to $\widetilde{P}^2(M)/\widetilde{G}^1$ represented by $x\mapsto\sigma_\nabla(x)=(x,\delta^i{}_j,-\Gamma^i{}_{jk})$.
Then, sections $\sigma_\nabla$ and $\sigma_{\nabla'}$ determine the same projective structure if and only if there is an $H^2$-valued function, say $a=(a^i{}_j,-(a^i{}_ja_k+a_ja^i{}_k))$ such that $\sigma_\nabla.a=\sigma_{\nabla'}$.
This condition is equivalent to that
\[
(x,a^i{}_j,-\Gamma^i{}_{lm}a^l{}_ja^m{}_k-(a^l{}_ja_k+a_ja^l{}_k))=(x,\delta^i{}_j,-\Gamma'{}^i{}_{jk}).
\]
holds in $\widetilde{P}^2(M)/\widetilde{G}^1$.
The left hand side is equal to $(x,\delta^i{}_j,-\Gamma^i{}_{jk}-(\delta^i{}_ja_k+\delta^i{}_ka_j))$.
Hence $\nabla$ and $\nabla'$ correspond to the same projective structure if and only if we have $\Gamma'{}^i{}_{jk}=\Gamma^i{}_{jk}+\delta^i{}_ja_k+\delta^i{}_ka_j$, that is, $\nabla$ and $\nabla'$ are projectively equivalent.
\end{proof}

\begin{remark}
Affine connections decide geodesics and hence projective structures.
The most standard projective structure is the one on $\R P^n$ and equivalences should be described in terms of linear fractional transformations even if we allow torsions.
This leads to above definitions.
Recall that projective structures are considered to be the same if they have the same (unparameterized) geodesics and the same torsions in this article.
\end{remark}

Let $\nabla$ be an affine connection.
We will describe the projective structure given by $\nabla$ and the associated normal projective connection.
For this purpose, we introduce the following

\begin{definition}
\label{def2.12}
Let $\nabla$ be an affine connection and $\{\Gamma^i{}_{jk}\}$ the Christoffel symbols with respect to a chart.
We define one-forms $\mu$ and $\nu$ by setting $\mu_j=\frac12(\Gamma^\alpha{}_{\alpha j}-\Gamma^\alpha{}_{j\alpha})$ and $\nu_j=-\frac1{2(n+1)}(\Gamma^\alpha{}_{\alpha j}+\Gamma^\alpha{}_{j\alpha})$.
We refer to $\mu$ as the \textit{reduced torsion} of $\nabla$.
\end{definition}

\begin{remark}
\begin{enumerate}
\item
The differential form $\Gamma^\alpha{}_{\alpha j}dx^j$ is the connection form of the connection on $\mathcal{E}(M)$ induced by $\nabla$.
The other differential form $\Gamma^\alpha{}_{k\alpha}dx^k$ also correspond to a connection on $\mathcal{E}(M)$.
These connections are the same if\/ $\nabla$ is torsion-free.
\item
The differential form $-\mu=-\mu_jdx^j$ is a kind of the Ricci tensor of the torsion.
\end{enumerate}
\end{remark}

Cartan connections can be found as follows.
\begin{lemma}
\label{lem2.21}
Let $(\omega^i,\omega^i{}_j,\omega_j)$ be a Cartan connection on $P$.
Let $\sigma\colon U\to P$ be a section, and set $\psi^i=\sigma^*\omega^i=\Pi^i{}_jdx^j$, $\psi^i{}_j=\sigma^*\omega^i{}_j=\Pi^i{}_{jk}dx^k$ and $\psi_j=\sigma^*\omega_j=\Pi_{jk}dx^k$.
Let $(a^i{}_j,a_j)$ be the coordinates for $H^2$ as in Definition~\ref{def2.1} and $(x^i,a^i{}_j,a_j)$ be the product coordinates for $P|_U\cong U\times H^2$, where the identification is given by $\sigma$.
If we set $(b^i{}_j)=(a^i{}_j)^{-1}$, then we have
\[
\renewcommand{\arraystretch}{1.5}
\begin{array}{l@{}c@{}l}
\omega^i&{}={}&b^i{}_\alpha\psi^\alpha\\*
&{}={}& b^i{}_\alpha\Pi^\alpha{}_\beta dx^\beta,\\*
\omega^i{}_j&{}={}&b^i{}_\alpha da^\alpha{}_j+b^i{}_\alpha\psi^\alpha{}_\beta a^\beta{}_j+b^i{}_\alpha\psi^\alpha a_j+\delta^i{}_ja_\alpha b^\alpha{}_\beta\psi^\beta\\*
&{}={}&b^i{}_\alpha da^\alpha{}_j+b^i{}_\alpha\Pi^\alpha{}_{\beta\gamma}a^\beta{}_jdx^\gamma+b^i{}_\alpha\Pi^\alpha{}_\beta a_jdx^\beta+\delta^i{}_ja_\alpha b^\alpha{}_\beta\Pi^\beta{}_\gamma dx^\gamma\\*
\omega_j&{}={}&da_j-a_\alpha b^\alpha{}_\beta da^\beta_j-a_\alpha b^\alpha{}_\beta\psi^\beta{}_\gamma a^\gamma{}_j+\psi_\alpha a^\alpha{}_j-a_\alpha b^\alpha{}_\beta\psi^\beta a_j\\*
&{}={}&da_j-a_\alpha b^\alpha{}_\beta da^\beta_j-a_\alpha b^\alpha{}_\beta\Pi^\beta{}_{\gamma\delta}a^\gamma{}_jdx^\delta+\Pi_{\alpha\beta}a^\alpha{}_jdx^\beta-a_\alpha b^\alpha{}_\beta\Pi^\beta{}_\gamma a_jdx^\gamma.
\end{array}
\]
\end{lemma}

Let $U$ be a chart of $M$ and $x^i$ the local coordinates on $U$.
Then, Proposition~17 of \cite{Kobayashi-Nagano} holds in the following form.

\begin{proposition}
\label{prop2.9}
Suppose that $n\geq2$ and let $\omega=(\omega^i,\omega^i{}_j,\omega_j)$ be the normal projective connection for the projective structure $P$ determined by $\nabla$.
Then, there is a unique section $\sigma\colon U\to P$ with the following properties\textup{:}
\begin{enumerate}[\textup{\theenumi)}]
\item
We have $\sigma^*\omega^i=dx^i$.
\item
If we set\/ $\sigma^*\omega^i{}_j=\Psi^i{}_j=\Pi^i{}_{jk}dx^k$, then we have $\Pi^i{}_{ik}=\mu_k$.
\end{enumerate}
We have moreover that
\begin{enumerate}
\item[\textup{2')}]
$\Pi^i{}_{ji}=-\mu_j$,
\end{enumerate}
and
\begin{alignat*}{3}
&\Pi^i{}_{jk}&{}={}&\Gamma^i{}_{jk}+\delta^i{}_j\nu_k+\delta^i{}_k\nu_j\\*
& &{}={}&\Gamma^i{}_{jk}-\frac1{2(n+1)}(\delta^i{}_j(\Gamma^\alpha{}_{\alpha k}+\Gamma^\alpha{}_{k\alpha})+\delta^i{}_k(\Gamma^\alpha{}_{\alpha j}+\Gamma^\alpha{}_{j\alpha})),\\*
&\Pi_{jk}&{}={}&\dfrac{-1}{n^2-1}\left(n\left(\pdif{\Pi^i{}_{jk}}{x^i}+\pdif{\mu_j}{x^k}-\mu_\alpha\Pi^\alpha{}_{jk}-\Pi^\alpha{}_{j\beta}\Pi^\beta{}_{\alpha k}\right)\right.\\*
& & &\hphantom{\frac{-1}{n^2-1}\biggl(}\left.+\left(\pdif{\Pi^i{}_{kj}}{x^i}+\pdif{\mu_k}{x^j}-\mu_\alpha\Pi^\alpha{}_{kj}-\Pi^\alpha{}_{k\beta}\Pi^\beta{}_{\alpha j}\right)\right),
\end{alignat*}
where $\{\Gamma^i{}_{jk}\}$ denote the Christoffel symbols and $\sigma^*\omega_j=\Psi_j=\Pi_{jk}dx^k$.
Finally, we can exchange conditions \textup{2)} and \textup{2')}.
\end{proposition}
\begin{proof}
Let $\sigma_0$ be the section from $M$ to $\widetilde{P}^2(M)/\widetilde{G}^1$ given by the connection, namely, $\sigma_0(x)=(x^i,\delta^i{}_j,-\Gamma^i{}_{jk})$.
Let $\overline{\sigma}_0$ denote the section from $M$ to $\widetilde{P}^2(M)/H^2$ induced by $\sigma_0$.
By the condition~1), $\sigma$ should be of the form $\overline{\sigma}_0.h$, where $h=(\delta^i{}_j,-(\delta^i{}_k\nu'_j+\delta^i{}_j\nu'_k))$ for some $\nu'_j$.
If $\sigma(x)=(x^i,\delta^i{}_j,-\Pi^i{}_{jk})$, then we have $\Pi^i{}_{jk}=\Gamma^i{}_{jk}+\delta^i{}_j\nu'_k+\delta^i{}_k\nu'_j$ (see Remark~\ref{rem2.5}).
Suppose that $\nu'_j$ can be so chosen that $\Pi^i{}_{ik}=\mu_k$ or $\Pi^i{}_{ji}=-\mu_j$.
Then, we accordingly have
\begin{align*}
\mu_k&=\Pi^i{}_{ik}=\Gamma^i{}_{ik}+(n+1)\nu'_k,\ \text{or}\\*
-\mu_j&=\Pi^i{}_{ji}=\Gamma^i{}_{ji}+(n+1)\nu'_j.
\end{align*}
The both conditions are equivalent to
\[
(n+1)\nu'_k=-\frac12(\Gamma^\alpha{}_{\alpha k}+\Gamma^\alpha{}_{k\alpha}).
\]
Hence we have $\nu'=\nu$ in the both cases.
The uniqueness also holds.
Conversely, if we define $\Pi^i{}_{jk}$ as in the statement and if we set $\sigma(x)=(x^i,\delta^i{}_j,-\Pi^i{}_{jk})$, then $\sigma$ induces a section to $\widetilde{P}^2(M)/H^2$ by Lemma~\ref{lem2.12} below.
We have $\sigma^*\omega^i=dx^i$ and $\sigma^*\omega^i{}_j=\Psi^i{}_j$.
If we set $\Psi_j=\sigma^*\omega_j$, then we have
\[
\tag{\thetheorem{}a}
\label{eq2.11}
\sigma^*\Omega^i{}_j=d\Psi^i{}_j+\Psi^i{}_k\wedge\Psi^k{}_j+dx^i\wedge\Psi_j-\delta^i{}_j\Psi_k\wedge dx^k.
\]
If we define $k^i{}_{jkl}$ by the conditions that $\sigma^*\Omega^i{}_j=\frac12k^i{}_{jkl}dx^k\wedge dx^l$ and $k^i{}_{jkl}+k^i{}_{jlk}=0$, then \eqref{eq2.11} is equivalent to
\[
k^i{}_{jkl}=\pdif{\Pi^i{}_{jl}}{x^k}-\pdif{\Pi^i{}_{jk}}{x^l}+\Pi^i{}_{\alpha k}\Pi^\alpha{}_{jl}-\Pi^i{}_{\alpha l}\Pi^\alpha{}_{jk}+\delta^i{}_k\Pi_{jl}-\delta^i{}_l\Pi_{jk}-\delta^i{}_j(\Pi_{lk}-\Pi_{kl}).
\]
Since $\omega$ is a normal projective connection, we have
\begin{align*}
0&=k^i{}_{jil}\\*
&=\pdif{\Pi^i{}_{jl}}{x^i}-\pdif{\Pi^i{}_{ji}}{x^l}+\Pi^i{}_{\alpha i}\Pi^\alpha{}_{jl}-\Pi^i{}_{\alpha l}\Pi^\alpha{}_{ji}+n\Pi_{jl}-\Pi_{jl}-(\Pi_{lj}-\Pi_{jl})\\*
&=\pdif{\Pi^i{}_{jl}}{x^i}+\pdif{\mu_j}{x^l}-\mu_\alpha\Pi^\alpha{}_{jl}-\Pi^i{}_{\alpha l}\Pi^\alpha{}_{ji}+n\Pi_{jl}-\Pi_{lj}.
\end{align*}
Regarding this equality as an equation with respect to $\Pi_{jk}$, we see that $\Pi_{jk}$ is given as in the statement.
\end{proof}

\begin{remark}
\label{rem2.17}
If we replace $\nu_j$ by $-\frac1{2(n+1)}(a\Gamma^\alpha{}_{\alpha j}+b\Gamma^\alpha{}_{j\alpha})$ in Definition~\ref{def2.12}, then Proposition~\ref{prop2.9} holds after replacing the conditions by
\begin{align*}
\Pi^\alpha{}_{\alpha k}&=\left(1-\frac{a}2\right)\Gamma^\alpha{}_{\alpha k}-\frac{b}2\Gamma^\alpha{}_{k\alpha},\\*
\Pi^\alpha{}_{j\alpha}&=-\frac{a}2\Gamma^\alpha{}_{\alpha k}+\left(1-\frac{b}2\right)\Gamma^\alpha{}_{k\alpha}.
\end{align*}
These are proportional to the reduced torsion if and only if $a+b=2$.
We choose $a=b=1$ as the simplest case, taking symmetricity into account.
The situation is similar in Theorem~\ref{thm6.13}.
\end{remark}

As in the classical case, we have the following.
We choose a branch of the logarithmic function in the complex category.

\begin{lemma}
\label{lem2.12}
Let $(U,\varphi)$ and $(\widehat{U},\widehat{\varphi})$ be charts.
We assume that $U=\widehat{U}$ and set $\psi=\widehat{\varphi}\circ\varphi^{-1}$.
If $\sigma$ and $\widehat{\sigma}$ denote the sections given by Proposition~\ref{prop2.9}, then we have
\[
\psi_*\sigma=\widehat\sigma.(a^i{}_j,-(a_ja^i{}_k+a_ka^i{}_j)),
\]
where $a^i{}_j=D\psi^i{}_j$ and $a_j=-\dfrac1{n+1}\pdif{\log J\psi}{x^j}$ with $J\psi=\det D\psi$.
\end{lemma}
\begin{proof}
We have
\[
\Gamma^i{}_{jk}=(D\psi^{-1})^i{}_\alpha H\psi^\alpha{}_{jk}+(D\psi^{-1})^i{}_\alpha\widehat{\Gamma}^\alpha{}_{\beta\gamma}D\psi^\beta{}_jD\psi^\gamma{}_k,
\]
where $D$ denotes the derivative and $H$ denotes the Hessian.
It follows that
\begin{align*}
\Pi^i{}_{jk}&=(D\psi^{-1})^i{}_\alpha H\psi^\alpha{}_{jk}+(D\psi^{-1})^i{}_\alpha\widehat{\Gamma}^\alpha{}_{\beta\gamma}D\psi^\beta{}_jD\psi^\gamma{}_k\\*
&\hphantom{{}={}}%
-\frac1{2(n+1)}\delta^i{}_j\left(\left(\pdif{\log J}{x^k}+\widehat{\Gamma}^\alpha{}_{\alpha\beta}D\psi^\beta{}_k\right)+\left(\pdif{\log J}{x^k}+\widehat{\Gamma}^\alpha{}_{\gamma\alpha}D\psi^\gamma{}_k\right)\right)\\*
&\hphantom{{}={}}%
-\frac1{2(n+1)}\delta^i{}_k\left(\left(\pdif{\log J}{x^j}+\widehat{\Gamma}^\alpha{}_{\alpha\beta}D\psi^\beta{}_j\right)+\left(\pdif{\log J}{x^j}+\widehat{\Gamma}^\alpha{}_{\gamma\alpha}D\psi^\gamma{}_j\right)\right)\\*
&=(D\psi^{-1})^i{}_\alpha H\psi^\alpha{}_{jk}+(D\psi^{-1})^i{}_\alpha\widehat{\Pi}^\alpha{}_{\beta\gamma}D\psi^\beta{}_jD\psi^\gamma{}_k\\*
&\hphantom{{}={}}%
-\frac1{n+1}\left(\delta^i{}_j\pdif{\log J}{x^k}+\delta^i{}_k\pdif{\log J}{x^j}\right),
\end{align*}
from which the lemma follows.
\end{proof}

If $\nabla$ is torsion-free, then $\Pi^i{}_{jk}$ and $\Pi_{jk}$ are well-known as follows~\cite{Kobayashi-Nagano}*{Proposition~17}, \cite{Roberts}*{Fundamental theorem for TW-connections}.

\begin{lemma}
\label{lem2.13}
If\/ $\nabla$ is torsion-free, then we have $\mu_j=0$ and $\Pi^i{}_{jk}=\Pi^i{}_{kj}$.
We have
\begin{align*}
\Pi^i{}_{jk}&=\Gamma^i{}_{jk}-\frac1{n+1}(\delta^i{}_j\Gamma^\alpha{}_{\alpha k}+\delta^i{}_k\Gamma^\alpha{}_{\alpha j}),\\*
\Pi_{jk}&=\Pi_{kj}=-\frac1{n-1}\left(\pdif{\Pi^i{}_{jk}}{x^i}-\Pi^\alpha{}_{j\beta}\Pi^\beta{}_{\alpha k}\right).
\end{align*}
Moreover, \/$\Omega^i{}_{ikl}=0$.
\end{lemma}
\begin{proof}
The first part is straightforward.
To show that $\Omega^i{}_j$ is trace-free, it suffices to show that $k^i{}_{ikl}=0$.
We have
\begin{align*}
k^i{}_{ikl}&=\pdif{\Pi^i{}_{il}}{x^k}-\pdif{\Pi^i{}_{ik}}{x^l}+\Pi^i{}_{\alpha k}\Pi^\alpha{}_{il}-\Pi^i{}_{\alpha l}\Pi^\alpha{}_{ik}+\Pi_{kl}-\Pi_{lk}-n(\Pi_{lk}-\Pi_{kl})\\*
&=\pdif{\mu_l}{x^k}-\pdif{\mu_k}{x^l}+(n+1)(\Pi_{kl}-\Pi_{lk})\\*
&=0.\qedhere
\end{align*}
\end{proof}

In this article, we are working with projective structures keeping torsion invariant.
If we allow to modify torsions, we have the following lemma and corollary~\cite{Weyl},~\cite{McKay}*{Lemma~11}.
We include a sketch of a proof for completeness.

\begin{lemma}
\label{lem2.25}
Let $\nabla$ and $\overline{\nabla}$ be connections of which the Christoffel symbols are $\{\Gamma^i{}_{jk}\}$ and $\{\overline{\Gamma}^i{}_{jk}\}$.
Then, the unparameterized geodesics of $\nabla$ and $\overline\nabla$ are the same if and only if\/ $\overline{\Gamma}^i{}_{jk}=\Gamma^i{}_{jk}+\delta^i{}_j\varphi_k+\delta^i{}_k\varphi_j+a^i{}_{jk}$, where $\{\varphi_k\}$ are components of a $1$-form of $M$, and $\{a^i{}_{jk}\}$ are components of $TM$-valued $2$-form on $M$ such that $a^i{}_{kj}=-a^i{}_{jk}$.
\end{lemma}
\begin{proof}
We follow the proof of~\cite{Kobayashi-Nagano}*{Proposition~12}.
We only show that the geodesic equation of $\nabla$ and $\overline{\nabla}$ are equivalent.
Let $s$ and $\overline{s}$ be parameters of geodesic of $\nabla$ and $\overline{\nabla}$, respectively.
Writing down the geodesic equation, we have
\begin{align*}
0&=\frac{d^2x^i}{d\overline{s}^2}+\overline{\Gamma}^i{}_{jk}\frac{dx^j}{d\overline{s}}\frac{dx^k}{d\overline{s}}\\*
&=\left(\frac{d^2x^i}{ds^2}+\Gamma^i{}_{jk}\frac{dx^j}{ds}\frac{dx^k}{ds}\right)\left(\frac{ds}{d\overline{s}}\right)^2+\frac{dx^i}{ds}\left(2\varphi_j\frac{dx^j}{d\overline{s}}\frac{ds}{d\overline{s}}+\frac{d^2s}{d\overline{s}^2}\right)+a^i{}_{jk}\frac{dx^j}{d\overline{s}}\frac{dx^k}{d\overline{s}}\\*
&=\left(\frac{d^2x^i}{ds^2}+\Gamma^i{}_{jk}\frac{dx^j}{ds}\frac{dx^k}{ds}\right)\left(\frac{ds}{d\overline{s}}\right)^2+\frac{dx^i}{ds}\left(2\varphi_j\frac{dx^j}{d\overline{s}}\frac{ds}{d\overline{s}}+\frac{d^2s}{d\overline{s}^2}\right),
\end{align*}
because $a^i{}_{kj}=-a^i{}_{jk}$.
Hence, it suffices to solve the equation $2\varphi_j\frac{dx^j}{d\overline{s}}\frac{ds}{d\overline{s}}+\frac{d^2s}{d\overline{s}^2}=0$.
\end{proof}

\begin{corollary}
\label{cor2.26}
Given an affine connection $\nabla$, we can find a torsion-free affine connection $\overline\nabla$ of which the geodesics are the same.
\end{corollary}
\begin{proof}
Let $T$ be the torsion of $\nabla$.
It suffices to set $\overline\nabla=\nabla+\frac12T$.
\end{proof}

\begin{remark}
\label{rem2.23}
A projective connection similar to the normal projective connection as in Theorem~\ref{thm5.3} is given by Hlavat\'y~\cite{Hlavaty}.
We refer to this connection as the \textup{Hlavat\'y connection}.
The components of the Hlavar\'y connection is given by
\begin{align*}
&\Phi^i{}_{jk}=\Gamma^i{}_{jk}+\frac1{n^2-1}(\delta^i{}_j(\Gamma^\alpha{}_{k\alpha}-n\Gamma^\alpha{}_{\alpha k})+\delta^i{}_k(\Gamma^\alpha{}_{\alpha j}-n\Gamma^\alpha{}_{j\alpha})).
\end{align*}
We have $\Phi^\alpha{}_{\alpha k}=0$ and $\Phi^\alpha{}_{j\alpha}=0$.
The Hlavat\'y connection can be obtained as follows.
First consider an affine connection $\overline{\nabla}$ of which the Christoffel symbols $\{\overline{\Gamma}{}^i{}_{jk}\}$ are given by
\[
\overline{\Gamma}{}^i{}_{jk}=\Gamma^i{}_{jk}-\frac1{n-1}(\delta^i{}_j\mu_k-\delta^i{}_k\mu_j).
\]
The geodesics of\/ $\nabla$ and $\overline\nabla$ are the same.
On the other hand, if $T$ and $\overline{T}$ denote the torsion of\/ $\nabla$ and $\overline\nabla$, then we have $\overline{T}{}^i{}_{jk}=T^i{}_{jk}+\frac2{n-1}(\delta^i{}_j\mu_k-\delta^i{}_k\mu_j)$.
We have
\begin{align*}
\overline\Gamma{}^\alpha{}_{\alpha k}&=\Gamma^\alpha{}_{\alpha k}-\mu_k=\frac12(\Gamma^\alpha{}_{\alpha k}+\Gamma^\alpha{}_{k\alpha})=-(n+1)\nu_k,\\*
\overline\Gamma{}^\alpha{}_{j\alpha}&=\Gamma^\alpha{}_{j\alpha}+\mu_j=\frac12(\Gamma^\alpha{}_{\alpha j}+\Gamma^\alpha{}_{j\alpha})=-(n+1)\nu_j.
\end{align*}
Hecne we have
\begin{align*}
\overline\mu_j&=\frac12(\overline\Gamma{}^\alpha{}_{\alpha j}-\overline\Gamma{}^\alpha{}_{j\alpha})=0,\\*
\overline\nu_j&=-\frac1{2(n+1)}(\overline\Gamma{}^\alpha{}_{\alpha j}+\overline\Gamma{}^\alpha{}_{j\alpha})=\nu_j.
\end{align*}
By some straightforward calculations, we see that $\overline\Pi{}^i{}_{jk}=\Phi^i{}_{jk}$.
Note that we have $\Phi^i{}_{jk}-\Pi^i{}_{jk}=\overline{\Gamma}{}^i{}_{jk}-\Gamma^i{}_{jk}=-\frac1{n-1}(\delta^i{}_j\mu_k-\delta^i{}_k\mu_j)$.
As $\overline{\mu}_j=0$, we have
\[
\overline\Pi{}_{jk}=\dfrac{-1}{n^2-1}\left(n\left(\pdif{\overline\Pi^i{}_{jk}}{x^i}-\overline\Pi^\alpha{}_{j\beta}\overline\Pi^\beta{}_{\alpha k}\right)+\left(\pdif{\overline\Pi^i{}_{kj}}{x^i}-\overline\Pi^\alpha{}_{k\beta}\overline\Pi^\beta{}_{\alpha j}\right)\right).
\]
\end{remark}

\section{Geodesics and completeness, flatness of projective structures}
Carefully examining arguments in~\cite{Kobayashi-Nagano}*{Sections~7 and 8}, we see that results presented there remain valid for projective structures with torsion.
We always consider equivalences in the sense of Definition~\ref{def2.8}, namely, we require the geodesics to be the same and also the torsions are the same.

As mentioned in the previous section, we have the following

\begin{proposition}[\cite{Weyl}, \cite{Kobayashi-Nagano}*{Proposition~12}]
Let $P$ be a projective structure of $M$ and $\nabla$ an affine connection which belongs to $P$.
If we disregard parametrizations, then geodesics of\/ $\nabla$ are geodesics of $P$ and vice versa.
\end{proposition}

\begin{definition}
\begin{enumerate}
\item
Let $M$ and $M'$ be manifolds with projective structures $P$ and $P'$.
A diffeomorphism $f\colon M\to M'$ is said to be a \textit{projective isomorphism} if $f_*\colon\widetilde{P}^2(M)\to\widetilde{P}^2(M')$ induces a bundle isomorphism from $P$ to $P'$.
\item
Let $M$ and $M'$ be manifolds with projective structures $P$ and $P'$.
A mapping $f\colon M\to M'$ is said to be a \textit{projective morphism} if for each $p\in M$, there exists an open neighborhood $U$ of $p$ such that the restriction of $f$ to $U$ is a projective isomorphism to its image.
\item
A projective structure $P$ on a manifold $M$ is said to be \textit{flat}, if for each $p\in M$, there exists an open neighborhood $U$ of $p$ and a projective isomorphism from $U$ to an open subset of $\R P^n$, where $n=\dim M$.
\end{enumerate}
\end{definition}

If a projective structure $P$ is flat, then the normal projective connection is torsion-free.
Hence we are in the classical settings so that we have the following.

\begin{theorem}[\cite{Kobayashi-Nagano}*{Theorem~15}]
A projective structure $P$ of a manifold $M$ is flat if and only if the torsion and the curvature of the normal projective connection vanish.
\end{theorem}

\begin{remark}
We also have estimates of the dimension of transformation groups which concern projective structures.
The results are parallel to Theorems~13 and 14 of\/ \cite{Kobayashi-Nagano}.
\end{remark}

\section{Thomas--Whitehead connections}

We follow arguments by Roberts~\cite{Roberts}.
Projective structures are described by means of connections on bundle of volumes.
Such connections are called Thomas--Whitehead connections.

\begin{definition}
Let $M$ be a manifold of dimension $n$.
If $M$ is orientable, then let $\mathcal{E}(M)$ be the principal $\R_{>0}$-bundle associated with $\bigwedge^nTM$.
If $M$ is non-orientable, we consider $\mathcal{E}(M)/\{\pm1\}$.
We equip an $\R$-action on $\mathcal{E}(M)$ by setting $va=ve^a$ for $v\in\mathcal{E}(M)$ and $a\in\R$.
We call $\mathcal{E}(M)$ the \textit{bundle of volume elements} over $M$.
\end{definition}

\begin{lemma}
The bundle of volume elements $\mathcal{E}(M)$ is a principal $\R$-bundle.
\end{lemma}
\begin{proof}
If $M$ is orientable, then we only consider charts compatible with the orientation.
Let $(U,\varphi)$ be a chart.
Then, $TM|_U$ is trivialized by $\left\{\pdif{}{x^i}\right\}$ so that $\mathcal{E}(M)|_U$ is trivialized by $\epsilon=\pdif{}{x^1}\wedge\cdots\wedge\pdif{}{x^n}$.
Indeed, if $p\in U$ and if $v_p\in\mathcal{E}_p(M)$, then we have $v_p=a\epsilon_p$ for some $a>0$.
Hence we can associate with $v_p$ a pair $(\epsilon_p,\log a)$.
In other words, the inverse of the mapping $(x^1,x^2,\ldots,x^n,x^{n+1})\mapsto(\varphi^{-1}(x^1,\ldots,x^n),\epsilon_{\varphi^{-1}(x^1,\ldots,x^n)}e^{x^{n+1}})$ is a local trivialization of $\mathcal{E}(M)$.
If $(\widehat{U},\widehat{\varphi})$ is another chart and if $\psi$ is the transition function from $U$ to $\widehat{U}$, then we have $\widehat{\epsilon}\det D\psi=\epsilon$.
Hence the transition function from $\mathcal{E}(M)|_U$ to $\mathcal{E}(M)|_{\widehat{U}}$ is given by $(p,t)\mapsto(p,t+\log\det D\psi)$ if $M$ is orientable and $(p,t)\mapsto(p,t+\log\norm{\det D\psi})$ if $M$ is non-orientable.
\end{proof}

\begin{remark}
In the complex category, we fix branches of the logarithms when choosing local trivializations.
\end{remark}

\begin{definition}
We locally set $\Psi=e^{-x^{n+1}}dx^1\wedge\cdots\wedge dx^n\wedge dx^{n+1}$ and call $\Psi$ the \textit{canonical positive odd density}.
\end{definition}

\begin{remark}
If $M$ is orientable, then $\Psi$ is indeed an $(n+1)$-form.
\end{remark}

\begin{definition}
For $a\in\R$ and $v\in\mathcal{E}(M)$, we set $R_av=v.a$.
Let $\mathrm{Lie}(\R)$ denote the Lie algebra of $\R$.
If $b\in\mathrm{Lie}(\R)$, then the vector field $X$ defined by $X_u=\left.\pdif{}{t}R_{bt}u\right|_{t=0}$ is called the \textit{fundamental vector field} associated with $b$.
In particular, the fundamental vector field associate with $1\in\mathrm{Lie}(\R)$ is called the \textit{canonical fundamental vector field} and denoted by $\xi$.
\end{definition}

We can reduce the definition of connection forms on $\mathcal{E}(M)$ as follows.

\begin{definition}
\label{def4.7}
A $\mathrm{Lie}(\R)$-valued $1$-form $\underline{\omega}$ on $\mathcal{E}(M)$ is called a \textit{connection form} if we have
\begin{enumerate}
\item
$\underline{\omega}(\xi)=1$, and
\item
$R_a{}^*\underline{\omega}=\Ad_{-a}\underline{\omega}=\underline{\omega}$ for $a\in\R$.
\end{enumerate}
\end{definition}

\begin{definition}
We set $\mathscr{F}=\left(\pdif{}{x^1},\ldots,\pdif{}{x^n},\pdif{}{x^{n+1}}\right)$ on $T\mathcal{E}(M)$.
\end{definition}

If $\psi$ denotes a change of coordinates, then the transition function is given by $\begin{pmatrix}
D\psi & 0\\
\partial\log J\psi & 1
\end{pmatrix}$, where $J\psi=\det D\psi$ and $\partial\log J\psi=\left(\pdif{\log J\psi}{x^1}\ \cdots\ \pdif{\log J\psi}{x^n}\right)$.

\begin{definition}[\cite{Roberts}, see also \cite{Thomas}]
\label{def1.8}
A \textit{Thomas--Whitehead projective connection}, or a \textit{TW-connection}, is a linear connection $\nabla$ on $T\mathcal{E}(M)$ with the following properties.
Let $\omega=(\omega^i{}_j)$ be the connection form of $\nabla$ with respect to $\mathscr{F}$.
\begin{enumerate}
\item
$\nabla\xi=-\frac1{n+1}\mathrm{id}$, namely, we have
\[
\omega^i{}_{n+1,j}=-\frac{\delta^i{}_j}{n+1},
\]
where $\delta^i{}_j=\begin{cases}
1, & i=j,\\
0, & i\neq j
\end{cases}$.
\item
We have $\omega^i{}_{j,n+1}=-\frac{\delta^i{}_j}{n+1}$.
\item
$R_{a*}(\nabla_XY)=\nabla_{R_{a*}X}(R_{a*}Y)$ for any $X,Y\in\mathfrak{X}(\mathcal{E}(M))$, namely, $\nabla$ is invariant under the right action of $\R$.
\end{enumerate}
We refer to $\nabla^{\underline{\omega}}$ as a \textit{TW-connection} on $TM$ induced by $\nabla$ and $\underline{\omega}$.
\end{definition}

\begin{remark}
TW-connections are usually assumed to be torsion-free.
In this case, the conditions 1) and 2) in Definition~\ref{def1.8} are equivalent.
\end{remark}

\begin{definition}
\label{defTW}
Let $\nabla$ be a TW-connection on $T\mathcal{E}(M)$ and $\underline{\omega}$ a connection form on $\mathcal{E}(M)$.
If $X,Y\in\mathfrak{X}(M)$, then let $\widetilde{X},\widetilde{Y}\in\mathfrak{X}(\mathcal{E}(M))$ be lifts of $X,Y$ horizontal with respect to $\underline{\omega}$.
We set
\[
\nabla^{\underline{\omega}}_XY=\pi_*\left(\nabla_{\widetilde{X}}\widetilde{Y}\right),
\]
where $\pi\colon\mathcal{E}(M)\to M$ is the projection.
\end{definition}

\begin{lemma}[see also Lemma~\ref{lem5.2}]
\label{Lem4.11}
$\nabla^{\underline{\omega}}$ is a connection on $TM$.
If\/ $\nabla$ is torsion-free, then so is $\nabla^{\underline{\omega}}$.
\end{lemma}
\begin{proof}
It is easy to see that $\nabla^{\underline{\omega}}$ is a connection.
If $\nabla$ is torsion-free, then we have
\begin{align*}
\nabla^{\underline{\omega}}_XY-\nabla^{\underline{\omega}}_YX&=\pi_*\left(\nabla_{\widetilde{X}}\widetilde{Y}-\nabla_{\widetilde{Y}}\widetilde{X}\right)\\*
&=\pi_*\left(\left[\widetilde{X},\widetilde{Y}\right]\right)\\*
&=\left[\pi_*\widetilde{X},\pi_*\widetilde{Y}\right]\\*
&=[X,Y].\qedhere
\end{align*}
\end{proof}

Let $\underline{\omega}$ be a connection form on $\mathcal{E}(M)$.
We locally have
\[
\underline{\omega}=f_1dx^1+\cdots+f_ndx^n+dx^{n+1}
\]
for some functions $f_1,\ldots,f_n$.

\begin{remark}
\begin{enumerate}[\textup{\theenumi)}]
\item
The functions $f_1,\ldots,f_n$ are independent of $x^{n+1}$ by 2) of Definition~\ref{def4.7}.
\item
Despite~1), $f_1dx^1+\cdots+f_ndx^n$ is not necessarily well-defined on $M$.
\end{enumerate}
\end{remark}

\begin{definition}
Let $e_i$ be the horizontal lift of $\pdif{}{x^i}$ to $T\mathcal{E}(M)$ with respect to $\underline{\omega}$, that is, we set
\[
e_i=\pdif{}{x^i}-f_i\pdif{}{x^{n+1}}.
\]
We set $e_{n+1}=\pdif{}{x^{n+1}}$ and $\mathscr{F}^H=(e_1,\ldots,e_n,e_{n+1})$.
\end{definition}

\begin{lemma}
\label{lem3.1}
Let $\psi$ be the transition function from $(x^1,\ldots,x^n)$ to $(\widehat{x}^1,\ldots,\widehat{x}^n)$.
We have
\[
\tag{\thetheorem{}a}
\label{eq5.1-1}
\left(\widehat{e}_1,\ldots,\widehat{e}_n,\widehat{e}_{n+1}\right)\begin{pmatrix}
D\psi \\
0 & 1
\end{pmatrix}=\left(e_1,\ldots,e_n,e_{n+1}\right).
\]
\end{lemma}
\begin{proof}
If we set $f=(f_1\ \cdots\ f_n)$, then we have
\[
\left(e_1,\ldots,e_n,e_{n+1}\right)\begin{pmatrix}
I_n\\
f & 1
\end{pmatrix}=\left(\pdif{}{x^1},\ldots,\pdif{}{x^n},\pdif{}{x^{n+1}}\right).
\]
If we set $J\psi=\det D\psi$, then we have
\begin{align*}
&\hphantom{{}={}}%
\left(\widehat{e}_1,\ldots,\widehat{e}_n,\widehat{e}_{n+1}\right)\begin{pmatrix}
I_n\\
\widehat{f} & 1
\end{pmatrix}\begin{pmatrix}
D\psi\\
D\log J\psi & 1
\end{pmatrix}\\*
&=\left(\pdif{}{\widehat{x}^1},\ldots,\pdif{}{\widehat{x}^n},\pdif{}{\widehat{x}^{n+1}}\right)\begin{pmatrix}
D\psi\\
D\log J\psi & 1
\end{pmatrix}\\*
&=\left(\pdif{}{x^1},\ldots,\pdif{}{x^n},\pdif{}{x^{n+1}}\right)\\*
&=\left(e_1,\ldots,e_n,e_{n+1}\right)\begin{pmatrix}
I_n\\
f & 1
\end{pmatrix}.
\end{align*}
On the other hand, if we set $dx={}^t(dx^1\ \cdots\ dx^n)$, then we have $\underline{\omega}=\begin{pmatrix}
f & 1
\end{pmatrix}\begin{pmatrix}
dx\\
dx^{n+1}
\end{pmatrix}$.
Hence we have
\[
\begin{pmatrix}
f & 1
\end{pmatrix}\begin{pmatrix}
dx\\
dx^{n+1}
\end{pmatrix}=\begin{pmatrix}
\widehat{f} &  1
\end{pmatrix}\begin{pmatrix}
d\widehat{x}\\
d\widehat{x}^{n+1}
\end{pmatrix}%
=\begin{pmatrix}
\widehat{f} & 1
\end{pmatrix}\begin{pmatrix}
D\psi \\
D\log J\psi & 1
\end{pmatrix}\begin{pmatrix}
dx\\
dx^{n+1}
\end{pmatrix}
\]
and consequently that
\[
\begin{pmatrix}
I_n\\
\widehat{f} & 1
\end{pmatrix}\begin{pmatrix}
D\psi\\
D\log J\psi & 1
\end{pmatrix}=\begin{pmatrix}
D\psi\\
f & 1
\end{pmatrix}=\begin{pmatrix}
D\psi\\
0 & 1
\end{pmatrix}\begin{pmatrix}
I_n\\
f & 1
\end{pmatrix}.
\]
Combining these equalities, we obtain the relation as desired.
\end{proof}

Let $\omega$ be the connection form of a TW-connection with respect to $\mathscr{F}$.
If we define $\omega'$ by the property
\[
\omega=\omega'-\frac1{n+1}\begin{pmatrix}
I_ndx^{n+1} & dx\\
0 & dx^{n+1}
\end{pmatrix},
\]
then $\omega'=\begin{pmatrix}
\alpha & 0\\
\beta & 0
\end{pmatrix}$, where $\alpha$ and $\beta$ do not involve $dx^{n+1}$.
Moreover, as $\nabla$ is invariant under the $\R$-action, $\alpha$ and $\beta$ projects to $M$.

\begin{remark}
The connection $\nabla$ is torsion-free if and only if we have $\alpha^i{}_{jk}=\alpha^i{}_{kj}$ and $\beta^i{}_{jk}=\beta^i{}_{kj}$.
\end{remark}

\begin{remark}
\label{rem4.17}
The transition rule of $\alpha$ and $\beta$ under changes of coordinates is given as follows.
We have
\begin{align*}
\omega&=\begin{pmatrix}
D\psi & 0\\
\partial\log J\psi & 1\end{pmatrix}^{-1}d\begin{pmatrix}
D\psi & 0\\
\partial\log J\psi & 1\end{pmatrix}+\begin{pmatrix}
D\psi & 0\\
\partial\log J\psi & 1\end{pmatrix}^{-1}\widehat\omega\begin{pmatrix}
D\psi & 0\\
\partial\log J\psi & 1\end{pmatrix}\\*
&=\begin{pmatrix}
(D\psi)^{-1}dD\psi & 0\\
-(\partial\log J\psi)(D\psi^{-1})dD\psi+d\partial\log J\psi & 0\end{pmatrix}\\*
&\hphantom{{}={}}%
+\begin{pmatrix}
(D\psi)^{-1}\widehat\alpha D\psi & 0\\
-(\partial\log J)(D\psi)^{-1}\widehat\alpha D\psi+\widehat\beta D\psi & 0
\end{pmatrix}\\*
&\hphantom{{}={}}%
-\frac1{n+1}\left(I_{n+1}d\widehat{x}^{n+1}+\begin{pmatrix}
(D\psi)^{-1}d\widehat{x}\partial\log J\psi & (D\psi)^{-1}d\widehat{x}\\
-(\partial\log J\psi)(D\psi^{-1})d\widehat{x}\partial\log J\psi & -\partial(\log J\psi)(D\psi)^{-1}d\widehat{x}
\end{pmatrix}\right)\\*
&=\begin{pmatrix}
(D\psi)^{-1}dD\psi & 0\\
-(\partial\log J\psi)(D\psi^{-1})dD\psi+d\partial\log J\psi & 0\end{pmatrix}\\*
&\hphantom{{}={}}%
+\begin{pmatrix}
(D\psi)^{-1}\widehat\alpha D\psi & 0\\
-(\partial\log J)(D\psi)^{-1}\widehat\alpha D\psi+\widehat\beta D\psi & 0
\end{pmatrix}\\*
&\hphantom{{}={}}%
-\frac1{n+1}\begin{pmatrix}
I_ndx^{n+1} & dx\\
0 & dx^{n+1}
\end{pmatrix}\\*
&\hphantom{{}={}}%
-\frac1{n+1}\begin{pmatrix}
I_nd\log J\psi+dx\partial\log J\psi & 0\\
-(d\log J\psi)\partial\log J\psi & 0
\end{pmatrix}.
\end{align*}

It follows that
\begin{align*}
\alpha&=(D\psi)^{-1}dD\psi-\frac1{n+1}(I_nd\log J\psi+dx\partial\log J\psi)+(D\psi)^{-1}\widehat\alpha D\psi,\\*
\beta&=-(\partial\log J\psi)(D\psi^{-1})dD\psi+d\partial\log J\psi+\frac1{n+1}(d\log J\psi)\partial\log J\psi\\*
&\hphantom{{}={}}%
-(\partial\log J\psi)(D\psi)^{-1}\widehat\alpha D\psi+\widehat\beta D\psi.
\end{align*}

Note that we have
\begin{align*}
\alpha^i{}_i&=\widehat{\alpha}^i{}_i,\\*
\beta&=-(\partial\log J\psi)\alpha-\frac1{n+1}((\partial\log J\psi)d\log J\psi+d\log J\psi(\partial\log J\psi))\\*
&\hphantom{{}={}}%
+\frac1{n+1}d\partial\log J\psi+\widehat\beta D\psi\\*
&=\frac1{n+1}(d\partial\log J\psi-2(\partial\log J\psi)d\log J\psi)-(\partial\log J\psi)\alpha+\widehat\beta D\psi.
\end{align*}
\end{remark}

\begin{remark}
\label{rem5.4}
If $\omega^H$ denotes the connection matrix of\/ $\nabla$ with respect to $\mathscr{F}^H$, then we have by the equality~\eqref{eq5.1-1} that
\begin{align*}
\omega^H&=\begin{pmatrix}
I_n\\
-f & 1\end{pmatrix}^{-1}d\begin{pmatrix}
I_n\\
-f & 1\end{pmatrix}+\begin{pmatrix}
I_n\\
-f & 1\end{pmatrix}^{-1}\omega\begin{pmatrix}
I_n\\
-f & 1\end{pmatrix}\\*
&=\begin{pmatrix}
O_n\\
-df & 0
\end{pmatrix}+\begin{pmatrix}
\alpha & 0 \\
f\alpha+\beta & 0
\end{pmatrix}\\*
&\hphantom{{}={}}%
-\frac1{n+1}\begin{pmatrix}
I_ndx^{n+1}-dx f & dx\\
fdx^{n+1}-(fdx+dx^{n+1})f & fdx+dx^{n+1}
\end{pmatrix}\\*
&=\begin{pmatrix}
\alpha+\frac1{n+1}(I_nfdx+dxf) & 0\\
-df+\frac1{n+1}fdxf+f\alpha+\beta & 0
\end{pmatrix}-\frac1{n+1}\begin{pmatrix}
I_n\underline{\omega} & dx\\
0 & \underline{\omega}
\end{pmatrix}.
\end{align*}
\end{remark}
Note that $(dx^1,\ldots,dx^n,\underline{\omega})$ is the dual to $\mathscr{F}^H$.

\begin{definition}
We set
\begin{align*}
\alpha^H&=\alpha+\frac1{n+1}(I_nfdx+dxf),\\*
\beta^H&=-df+\frac1{n+1}fdxf+f\alpha+\beta.
\end{align*}
\end{definition}

We have the following

\begin{lemma}
\label{lem5.2}
The connection form of\/ $\nabla^{\underline{\omega}}$ with respect to $\left(\pdif{}{x^1},\ldots,\pdif{}{x^n}\right)$ is equal to $\alpha^H$.
Indeed, we have
\begin{align*}
\alpha^H&=D\psi^{-1}dD\psi+D\psi^{-1}\widehat{\alpha}^HD\psi,\\*
\beta^H&=\widehat{\beta}^HD\psi.
\end{align*}
\end{lemma}
\begin{proof}
The first part follows directly from Definition~\ref{defTW}.
Let $(U,\varphi)$ and $(\widehat{U},\widehat{\varphi})$ be charts, and $\omega^H$ and $\widehat{\omega}^H$ connection forms of $\nabla$ with respect to $\mathscr{F}^H$ and $\widehat{\mathscr{F}}^H$, respectively.
Then, by Lemma~\ref{lem3.1}, we have
\begin{align*}
\omega^H&=\begin{pmatrix}
D\psi\\
& 1
\end{pmatrix}^{-1}d\begin{pmatrix}
D\psi\\
& 1
\end{pmatrix}+\begin{pmatrix}
D\psi\\
& 1
\end{pmatrix}^{-1}\widehat{\omega}^H\begin{pmatrix}
D\psi\\
& 1
\end{pmatrix}\\*
&=\begin{pmatrix}
D\psi^{-1}dD\psi & 0\\
0 & 0
\end{pmatrix}+\begin{pmatrix}
D\psi^{-1}\widehat\alpha^HD\psi & 0\\
\widehat\beta^HD\psi & 0
\end{pmatrix}-\frac1{n+1}\begin{pmatrix}
I_n\underline{\omega} & D\psi^{-1}d\widehat{x}\\
0 & \underline{\omega}
\end{pmatrix}\\*
&=\begin{pmatrix}
D\psi^{-1}dD\psi+D\psi^{-1}\widehat\alpha^HD\psi & 0\\
\widehat\beta^HD\psi & 0
\end{pmatrix}-\frac1{n+1}\begin{pmatrix}
I_n\underline{\omega} & dx\\
0 & \underline{\omega}
\end{pmatrix}.\qedhere
\end{align*}
\end{proof}

\begin{theorem}
\label{thm1.15}
If\/ $\nabla$ is a TW-connection on $T\mathcal{E}(M)$ and if $\underline{\omega}$ and $\underline{\omega}'$ are connection forms on $\mathcal{E}(M)$, then
\begin{enumerate}[\textup{\theenumi)}]
\item
$\underline{\omega}'-\underline{\omega}=\pi^*\rho$ for some $1$-form $\rho$ on $M$, and
\item
We have
\[
\nabla^{\underline{\omega}'}-\nabla^{\underline{\omega}}=\frac1{n+1}\rho\otimes\mathrm{id}+\frac1{n+1}\mathrm{id}\otimes\rho.
\]
\item
$\nabla^{\underline{\omega}}$ and $\nabla^{\underline{\omega}'}$ are projectively equivalent.
\end{enumerate}
\end{theorem}
\begin{proof}
First, we have $\omega'(\xi)-\omega(\xi)=0$ and $R_a{}^*(\omega'-\omega)=\omega'-\omega$.
Hence we have $\omega'-\omega=\pi^*\rho$ for some $1$-form on $M$.
2) follows from Remark~\ref{rem5.4} and Lemma~\ref{lem5.2}.
3) follows from 2) and Lemma~\ref{lem1.12}.
\end{proof}

\begin{theorem}
Fix a TW-connection $\nabla$ on $T\mathcal{E}(M)$ and a connection form $\underline{\omega}$ on $\mathcal{E}(M)$.
Then, there is a one-to-one correspondence between the set of connection forms on $\mathcal{E}(M)$ and the set of linear connections in the projective equivalence class represented by $\nabla^{\underline{\omega}}$.
\end{theorem}
\begin{proof}
Let $\mathcal{D}$ be a linear connection projectively equivalent to $\nabla^{\underline{\omega}}$.
There is a $1$-form $\rho$ such that $\mathcal{D}-\nabla^{\underline{\omega}}=\frac1{n+1}\rho\otimes\mathrm{id}+\frac1{n+1}\mathrm{id}\otimes\rho$.
If we set $\underline{\omega}'=\omega+\pi^*\rho$, then we have $\nabla^{\underline{\omega}'}=\mathcal{D}$ by Theorem~\ref{thm1.15}.
Suppose conversely that $\nabla^{\underline{\omega}_1}=\nabla^{\underline{\omega}_2}$.
Then $\underline{\omega}_1=\underline{\omega_2}$ by Lemma~\ref{lem1.14}.
\end{proof}

\begin{definition}
If $\omega$ is a $\gl_n(\R)$-valued $1$-form, then we set $R(\omega)=d\omega+\omega\wedge\omega$.
\end{definition}
Needless to say that $R(\omega)$ is the curvature form if $\omega$ is a connection form of a linear connection.

\begin{lemma}
\label{lem4.18}
The curvature form of a TW-connection with respect to $\mathscr{F}$ is given by
\[
R(\omega)=d\omega+\omega\wedge\omega=\begin{pmatrix}
d\alpha+\alpha\wedge\alpha-\frac1{n+1}dx\wedge\beta & -\frac1{n+1}\alpha\wedge dx\\
d\beta+\beta\wedge\alpha & -\frac1{n+1}\beta\wedge dx
\end{pmatrix}.
\]
The TW-connection is torsion free if and only if $\alpha\wedge dx=0$ and $\beta\wedge dx=0$,
\end{lemma}
\begin{proof}
We have
\begin{align*}
&\hphantom{{}={}}d\omega+\omega\wedge\omega\\*
&=d\omega'+\omega'\wedge\omega'\\*
&\hphantom{{}={}}-\frac1{n+1}\omega'\wedge\begin{pmatrix}
I_ndx^{n+1} & dx\\
0 & dx^{n+1}
\end{pmatrix}-\frac1{n+1}\begin{pmatrix}
I_ndx^{n+1} & dx\\
0 & dx^{n+1}
\end{pmatrix}\wedge\omega'\\*
&\hphantom{{}={}}+\frac1{(n+1)^2}\begin{pmatrix}
I_ndx^{n+1} & dx\\
0 & dx^{n+1}
\end{pmatrix}\wedge\begin{pmatrix}
I_ndx^{n+1} & dx\\
0 & dx^{n+1}
\end{pmatrix}\\*
&=\begin{pmatrix}
d\alpha+\alpha\wedge\alpha & 0\\
d\beta+\beta\wedge\alpha & 0
\end{pmatrix}-\frac1{n+1}\begin{pmatrix}
\alpha\wedge dx^{n+1}+dx^{n+1}\wedge\alpha+dx\wedge\beta & \alpha\wedge dx\\
\beta\wedge dx^{n+1}+dx^{n+1}\wedge\beta & \beta\wedge dx
\end{pmatrix}\\*
&=\begin{pmatrix}
d\alpha+\alpha\wedge\alpha & 0\\
d\beta+\beta\wedge\alpha & 0
\end{pmatrix}-\frac1{n+1}\begin{pmatrix}
dx\wedge\beta & \alpha\wedge dx\\
0 & \beta\wedge dx
\end{pmatrix}.
\end{align*}
If $\nabla$ is torsion-free, then we have $(\alpha\wedge dx)^i=\alpha^i{}_{jk}dx^k\wedge dx^j=0$.
Similarly, we have $\beta\wedge dx=0$.
The converse is easy.
\end{proof}

In view of Definition~\ref{def1.2}, we introduce the following

\begin{definition}
We regard the curvature form $d\omega+\omega\wedge\omega$ as being valued in $\mathfrak{pgl}_{n+1}(\R)=\mathfrak{m}\oplus\mathfrak{gl_n}(\R)\oplus\mathfrak{m}^*$, and represent the curvature form as $(\rho^i,\rho^i{}_j,\rho_j)$.
We call $\rho^i$ the \textit{torsion} and $(\rho^i{}_j,\rho_j)$ the \textit{curvature} of $\nabla$ as a projective connection.
\end{definition}

\begin{lemma}
\label{lem5.4}
We have
\begin{alignat*}{3}
&\rho^i&{}={}&-\frac1{n+1}\alpha\wedge dx,\\*[-2pt]
&\rho^i{}_j&{}={}&d\alpha+\alpha\wedge\alpha-\frac1{n+1}(dx\wedge\beta-\beta\wedge dxI_n),\\*
& &{}={}&d\alpha+\alpha\wedge\alpha+dx\wedge\beta'-\beta'\wedge dxI_n,\\*
&\rho_j&{}={}&d\beta+\beta\wedge\alpha\\*
& &{}={}&-(n+1)(d\beta'+\beta'\wedge\alpha),
\end{alignat*}
where $\beta'=-\frac1{n+1}\beta$.
\end{lemma}

\begin{definition}
We define the Ricci curvature $\mathop{\mathrm{Ric}}(\nabla)$ of a TW-connection $\nabla$ by
\begin{align*}
&\hphantom{{}={}}%
\mathop{\mathrm{Ric}}(\nabla)_{jk}\\*
&=\rho^i{}_{jik}\\*
&=\pdif{\alpha^i{}_{jk}}{x^i}-\pdif{\alpha^i{}_{ji}}{x^k}+\alpha^i{}_{\gamma i}\alpha^\gamma{}_{jk}-\alpha^i{}_{\gamma k}\alpha^\gamma{}_{ji}%
-\frac1{n+1}(n\beta_{jk}-\beta_{jk}+\beta_{jk}-\beta_{kj})\\*
&=\pdif{\alpha^i{}_{jk}}{x^i}-\pdif{\alpha^i{}_{ji}}{x^k}+\alpha^i{}_{\gamma i}\alpha^\gamma{}_{jk}-\alpha^i{}_{\gamma k}\alpha^\gamma{}_{ji}%
-\frac1{n+1}(n\beta_{jk}-\beta_{kj}).
\end{align*}
\end{definition}

The fundamental theorem for TW-connections by Roberts~\cite{Roberts} holds in the following form in the present setting.

\begin{theorem}
\label{thm6.13}
Suppose that $\dim M\geq2$ and a projective structure is of $M$ is given by an affine connection $\nabla_M$.
Let $\Psi_M$ be the canonical positive odd scalar density on $\mathcal{E}(M)$ and $\mu_M$ the reduced torsion of\/ $\nabla_M$ regarded as a form on $\mathcal{E}(M)$ by pull-back.
Then, there exists a unique TW-connection $\nabla$ such that
\begin{enumerate}[\textup{\theenumi)}]
\item
$\nabla\Psi_M=-\mu_M\otimes\Psi_M$.
\item
$\nabla$ is Ricci-flat.
\item
$\nabla$ induces the given projective equivalence class on $M$.
\end{enumerate}
Moreover, there is a unique connection on $\mathcal{E}(M)$ such that $\alpha^H$ is the connection form of\/ $\nabla_M$ with respect to~$\left(\pdif{}{x^i}\right)_{1\leq i\leq n}$.

Indeed, if $\{\Gamma^i{}_{jk}\}$ denotes the Christoffel symbols of\/ $\nabla_M$, then we have
\begin{align*}
\alpha^i{}_{jk}&=\Gamma^i{}_{jk}-\frac1{2(n+1)}(\delta^i{}_k(\Gamma^a{}_{aj}+\Gamma^a{}_{ja})+\delta^i{}_j(\Gamma^a{}_{ak}+\Gamma^a{}_{ka})),\\*
\beta_{jk}&=\dfrac{1}{n-1}\left(n\left(\pdif{\alpha^i{}_{jk}}{x^i}+\pdif{\mu_M{}_j}{x^k}-\mu_M{}_a\alpha^a{}_{jk}-\alpha^a{}_{jb}\alpha^b{}_{ak}\right)\right.\\*
&\hphantom{\Pi_{jk}{}={}\frac{-1}{n-1}\biggl(}\left.+\left(\pdif{\alpha^i{}_{kj}}{x^i}+\pdif{\mu_M{}_k}{x^j}-\mu_M{}_a\alpha^a{}_{kj}-\alpha^a{}_{kb}\alpha^b{}_{aj}\right)\right),
\end{align*}
where
\[
\begin{pmatrix}
\alpha & 0\\
\beta & 0
\end{pmatrix}-\frac1{n+1}\begin{pmatrix}
I_ndx^{n+1} & dx\\
0 & dx^{n+1}
\end{pmatrix}
\]
is the connection matrix of\/ $\nabla$ with respect to $\mathscr{F}$.
The connection of $\mathcal{E}(M)$ is given by $\underline{\omega}=\frac12(\Gamma^\alpha{}_{\alpha j}+\Gamma^\alpha{}_{j\alpha})$.
\end{theorem}
\begin{proof}
Let $\{\Gamma^i{}_{jk}\}$ be the Christoffel symbols of $\nabla_M$ and set $\alpha^H=(\Gamma^i{}_{jk}dx^k)$.
If we fix a connection $\underline{\omega}=fdx+dx^{n+1}$ on $\mathcal{E}(M)$, then a TW-connection is given by $\begin{pmatrix}
\alpha^H-\frac1{n+1}(I_nfdx+dxf) & 0\\
df+\frac1{n+1}fdx f-f\alpha^H+\beta^H & 0
\end{pmatrix}-\dfrac1{n+1}\begin{pmatrix}
I_ndx^{n+1} & dx\\
0 & dx^{n+1}
\end{pmatrix}$, where $\beta^H$ is an $\mathfrak{m}^*$-valued $1$-form (see Remark~\ref{rem5.4}).
Note that even if we replace $\nabla$ by a projectively equivalent connection, then $\underline\omega$ is modified while the TW-connection remains in the same form.
We have
\[
\nabla\Psi_M=(-(\alpha^H)^\alpha{}_{\alpha j}dx^j+f_jdx^j)\otimes\Psi_M=(-\Gamma^\alpha{}_{\alpha j}dx^j+f_jdx^j)\otimes\Psi_M.
\]
By the condition~1), we have $\Gamma^\alpha{}_{\alpha j}-f_j=\frac12(\Gamma^\alpha{}_{\alpha j}-\Gamma^\alpha{}_{j\alpha})$ so that
\[
f_j=\frac12(\Gamma^\alpha{}_{\alpha j}+\Gamma^\alpha{}_{j \alpha}).
\]
If set $\alpha=\alpha^H-\frac1{n+1}(I_nfdx+dxf)$ and $\beta=df+\frac1{n+1}fdx f-f\alpha^H+\beta^H$, then we have by the condition~2) that
\[
\pdif{\alpha^i{}_{jk}}{x^i}-\pdif{\alpha^i{}_{ji}}{x^k}+\alpha^i{}_{\gamma i}\alpha^\gamma{}_{jk}-\alpha^i{}_{\gamma k}\alpha^\gamma{}_{ji}-\frac1{n+1}(n\beta_{jk}-\beta_{kj})=0.
\]
It follows that $\beta_{jk}$ are given as in the statement.
Conversely, if we define $\alpha^i{}_{jk}$ and $\beta_{jk}$ as in the statement, then $\nabla$ is a TW-connection with the required properties.
Since $\alpha^i{}_{jk}$ and $\beta_{jk}$ are independent of $\underline\omega$, $\nabla$ is unique.
\end{proof}

It is natural to introduce the following

\begin{definition}
We call the TW-connection given by Theorem~\ref{thm6.13} the \textit{normal TW-connection}.
\end{definition}

\begin{remark}
If we only require uniqueness of normal TW-connections, then we can modify the normalizing conditions 1) and 2) in Theorem~\ref{thm6.13} by similar reasons as in Remark~\ref{rem2.17}.
The conditions are so chosen that components of the normal TW-connections coincide with the normal Cartan projective connections up to multiplication of constants.
Actually, $\alpha^i{}_{jk}$ and $\beta'{}_{jk}$ coincide with $\Pi^i{}_{jk}$ and $\Pi_{jk}$ given by Proposition~\ref{prop2.9}.
\end{remark}

\begin{remark}
Suppose that the projective structure in Theorem~\ref{thm6.13} is torsion-free.
Then, $\nabla_M$ is always torsion-free so that the condition~1) reduces to $\nabla\Psi_M=0$, which is independent of\/ $\nabla_M$.
In addition, we have
\begin{align*}
\alpha^i{}_{jk}&=\Gamma^i{}_{jk}-\frac1{n+1}(\delta^i{}_k\Gamma^a{}_{aj}+\delta^i{}_j\Gamma^a{}_{ak}),\\*
\beta_{jk}&=\dfrac{n+1}{n-1}\left(\pdif{\alpha^i{}_{jk}}{x^i}-\alpha^a{}_{jb}\alpha^b{}_{ak}\right).
\end{align*}
\end{remark}

\begin{remark}
If we allow to modify the torsion keeping the geodesics, then we can uniquely find a TW-connection which corresponds to the Hlavat\'y connection \textup{(}Remark~\ref{rem2.23}\textup{)}.
We can also uniquely find a TW-connection which corresponds to the connection of which the Christoffel symbols are $\left\{\frac12(\Gamma^i{}_{jk}+\Gamma^i{}_{kj})\right\}$.
\end{remark}

\begin{remark}
Projective connections, especially TW-connections are used for studies of transversal projective structures of foliations~\cite{asuke:2015}, \cite{asuke:2017}.
For this purpose, it is natural not to assume projective structures to be torsion-free.
Indeed, usually we begin with Bott connections on normal bundles of foliations, which are often with torsions.
If we consider projective connections related with Bott connection, they are usually with torsions.
\end{remark}

\section{Structural equivalences of TW-connections}
We continue to follow the arguments in \cite{Roberts}.

\begin{definition}[\cite{Roberts2}]
TW-connections $\nabla$ and $\nabla'$ are said to be \textit{structurally equivalent} if $\nabla$ and $\nabla'$ induce the same projective structure.
\end{definition}

\begin{theorem}
\label{thm1.18}
TW-connections $\nabla$ and $\nabla'$ are structurally equivalent if and only if there is a $(0,2)$-tensor $\beta$ on $\mathcal{E}(M)$ such that
\[
\tag{\thetheorem{}a}
\label{eq1.19}
\left\{\begin{aligned}
&L_\xi\beta=0,\\*
&\beta(\xi,\xi)=0,
\end{aligned}
\right.
\]
and
\[
\tag{\thetheorem{}b}
\label{eq1.19-3}
\nabla'=\nabla+(\iota'_\xi\beta)\otimes\mathrm{id}+\mathrm{id}\otimes(\iota'_\xi\beta)-\beta\otimes\xi,
\]
where $\iota'_\xi\beta=\beta(\;\cdot\;,\xi)$.
Such a $\beta$ is unique.
If\/ $\nabla$ and $\nabla'$ are torsion-free, then $\beta$ is symmetric.
\end{theorem}

Before proving Theorem~\ref{thm1.18}, we show the following
\begin{lemma}
If the condition \eqref{eq1.19} holds, then there is a $1$-form $\overline\beta$ on $M$ such that $\iota'_\xi\beta=\pi^*\overline\beta$.
\end{lemma}
\begin{proof}
Let $\iota$ be the usual inner product.
We locally represent $\beta$ as $\beta=\beta_{ij}dx^i\otimes dx^j$.
We have $\iota'_\xi\beta=\beta_{i,n+1}dx^i$.
On the other hand, we have $0=L_\xi\beta=\pdif{\beta_{ij}}{x^{n+1}}dx^i\otimes dx^j$.
Hence we have $\iota_\xi(\iota'_\xi\beta)=0$ and $\iota_\xi d(\iota'_\xi\beta)=\pdif{\beta_{i,n+1}}{x^{n+1}}dx^i-\pdif{\beta_{n+1,n+1}}{x^j}dx^j=\iota'_\xi(L_\xi\beta)=0$.
\end{proof}

\begin{remark}
\label{rem1.21}
If \eqref{eq1.19-3} holds and if $\underline{\omega}$ is a connection form on $\mathcal{E}(M)$, then we have
\[
\nabla'{}^{\underline{\omega}}=\nabla^{\underline{\omega}}+\overline{\beta}\otimes\mathrm{id}+\mathrm{id}\otimes\overline{\beta}.
\]
\end{remark}

\begin{proof}[Proof of Theorem \textup{\ref{thm1.18}}]
The proof is essentially identical to that of Theorem~3.6 in~\cite{Roberts}.
Keep in mind that connections need not be torsion-free.
First assume that there exists a $\beta$ which satisfy \eqref{eq1.19} and \eqref{eq1.19-3}.
If we set
\[
\widehat\nabla=\nabla+(\iota'_\xi\beta)\otimes\mathrm{id}+\mathrm{id}\otimes(\iota'_\xi\beta)-\beta\otimes\xi,
\]
then $\widehat\nabla$ is a TW-connection.
Note that $\beta$ is invariant under the $\R$-action because $L_\xi\beta=0$.
Let now $\underline{\omega}$ be a connection form on $\mathcal{E}(M)$ and $X,Y\in\mathfrak{X}(M)$.
If $\widetilde{X}$ and $\widetilde{Y}$ denote horizontal lifts of $X$ and $Y$, then we have
\[
\widehat\nabla_{\widetilde{X}}\widetilde{Y}=\nabla_{\widetilde{X}}\widetilde{Y}+\pi^*\overline{\beta}(\widetilde{X})\widetilde{Y}+\pi^*\overline{\beta}(\widetilde{Y})\widetilde{X}-\beta(\widetilde{X},\widetilde{Y})\xi
\]
for some $1$-form $\overline{\beta}$ on $M$.
Hence we have
\[
\tag{\ref{thm1.18}c}
\label{eq1.20}
\widehat\nabla^{\underline{\omega}}{}_XY=\nabla^{\underline{\omega}}{}_XY+\overline\beta(X)Y+\overline{\beta}(Y)X,
\]
which means that $\nabla^{\underline{\omega}}$ and $\widehat\nabla^{\underline{\omega}}$ are projectively equivalent.
Hence $\nabla$ and $\widehat{\nabla}$ are structurally equivalent.
Suppose conversely that $\nabla$ and $\widehat\nabla$ are structurally equivalent.
If we fix a connection form $\underline{\omega}$, then
\[
\widehat\nabla^{\underline{\omega}}=\nabla^{\underline{\omega}}+\overline{\beta}\otimes\mathrm{id}+\mathrm{id}\otimes\overline{\beta}
\]
for some $1$-form $\overline{\beta}$ on $M$.
We set, for $\widetilde{X},\widetilde{Y}\in\mathfrak{X}(\mathcal{E}(M))$,
\[
\beta(\widetilde{X},\widetilde{Y})=\omega(\nabla_{\widetilde{X}}\widetilde{Y}-\widehat{\nabla}_{\widetilde{X}}\widetilde{Y})+\pi^*\overline{\beta}(\widetilde{X})\omega(\widetilde{Y})+\pi^*\overline\beta(\widetilde{Y})\omega(\widetilde{X}).
\]
It is clear that $\beta$ is a $(0,2)$-tensor.
We have $L_\xi\beta=0$ and $\beta(\xi,\xi)=0$ because $\nabla$ and $\widehat\nabla$ are TW-connections.
If in addition $\nabla$ and $\nabla'$ are torsion-free, then $\beta$ is symmetric.
We will show that the equality \eqref{eq1.19-3} holds.
Let $\widetilde{X},\widetilde{Y}\in\mathfrak{X}(\mathcal{E}(M))$.
First assume that $\widetilde{X}$ and $\widetilde{Y}$ are horizontal lifts of $X,Y\in\mathfrak{X}(M)$.
Then, the equality \eqref{eq1.20} holds.
If $\widetilde{\nabla^{\underline{\omega}}{}_XY}$ and $\widetilde{\widehat{\nabla}^{\underline{\omega}}{}_XY}$ denote the horizontal lifts of $\nabla^{\underline{\omega}}{}_XY$ and $\widehat{\nabla}^{\underline{\omega}}{}_XY$, then we have
\begin{align*}
\nabla_{\widetilde{X}}\widetilde{Y}&=\widetilde{\nabla^{\underline{\omega}}{}_XY}+\omega(\nabla_{\widetilde{X}}\widetilde{Y})\xi,\\*
\widehat{\nabla}_{\widetilde{X}}\widetilde{Y}&=\widetilde{\widehat{\nabla}^{\underline{\omega}}{}_XY}+\omega(\widehat{\nabla}_{\widetilde{X}}\widetilde{Y})\xi.
\end{align*}
It follows that
\begin{align*}
\widehat{\nabla}_XY&=\widetilde{\widehat{\nabla}^{\underline{\omega}}{}_XY}+\omega(\widehat{\nabla}_XY)\xi\\
&=\widetilde{\nabla^{\underline{\omega}}{}_XY}+\pi^*\overline\beta(\widetilde{X})\widetilde{Y}+\pi^*\overline\beta(\widetilde{Y})\widetilde{X}+\omega(\widehat{\nabla}_XY)\xi\\*
&=\nabla_XY-\omega(\nabla_XY)\xi+\pi^*\overline\beta(\widetilde{X})\widetilde{Y}+\pi^*\overline\beta(\widetilde{Y})\widetilde{X}+\omega(\widehat{\nabla}_XY)\xi.
\end{align*}
On the other hand, we have
\begin{align*}
\iota'_\xi\beta(\widetilde{X})&=\omega(\nabla_{\widetilde{X}}\xi-\widehat{\nabla}_{\widetilde{X}}\xi)+\pi^*\overline\beta(\widetilde{X})\\*
&=\overline\beta(X).
\end{align*}
Similarly, we have $\iota'_\xi\beta(\widetilde{Y})=\overline{\beta}(Y)$.
Hence we have
\begin{align*}
\widehat{\nabla}_XY&=\nabla_XY-\omega(\nabla_XY)\xi+\pi^*\overline\beta(\widetilde{X})\widetilde{Y}+\pi^*\overline\beta(\widetilde{Y})\widetilde{X}+\omega(\widehat{\nabla}_XY)\xi\\*
&=\nabla_XY+\iota'_\xi\beta(\widetilde{X})\widetilde{Y}+\iota'_\xi\beta(\widetilde{Y})\widetilde{X}+\omega(\widehat{\nabla}_XY-\nabla_XY)\xi\\*
&=\nabla_XY+\iota'_\xi\beta(\widetilde{X})\widetilde{Y}+\iota'_\xi\beta(\widetilde{Y})\widetilde{X}-\beta(\widehat{X},\widehat{Y})\xi.
\end{align*}
Next, we assume that $\widetilde{Y}=\xi$.
We have $\beta(\widetilde{X},\widetilde{Y})=\pi^*\overline\beta(\widetilde{X})$ so that
\begin{align*}
&\hphantom{{}={}}%
\nabla_{\widetilde{X}}\xi+\iota'_\xi\beta(\widetilde{X})\xi+\iota'_\xi\beta(\xi)\widetilde{X}-\beta(\widetilde{X},\xi)\xi\\*
&=-\frac1{n+1}\widetilde{X}\\*
&=\widehat{\nabla}_{\widetilde{X}}\xi.
\end{align*}
We assume lastly that $\widetilde{X}=\xi$.
We have
\begin{align*}
&\hphantom{{}={}}%
\nabla_\xi\widetilde{Y}+\iota'_\xi\beta(\xi)\widetilde{Y}+\iota'_\xi\beta(\widetilde{Y})\xi-\beta(\xi,\widetilde{Y})\xi\\*
&=\nabla_\xi\widetilde{Y}+\iota'_\xi\beta(\widetilde{Y})\xi-\omega(\nabla_\xi\widetilde{Y}-\widehat{\nabla}_\xi\widetilde{Y})\xi-\pi^*\overline{\beta}(\widetilde{Y})\xi\\*
&=\widehat{\nabla}_\xi\widetilde{Y}.
\end{align*}
Therefore, the equality \eqref{eq1.19-3} holds.
Finally, suppose that $\beta'$ also satisfy the equalities \eqref{eq1.19} and \eqref{eq1.19-3} if we replace $\beta$ with $\beta'$.
Then we have $\iota'_\xi\beta=\overline{\beta}$ and $\iota'_\xi\beta'=\overline{\beta}'$ for some $1$-forms $\beta$ and $\beta'$.
By Remark \ref{rem1.21}, we have $\overline\beta=\overline\beta'$.
On the other hand, we have
\begin{align*}
\nabla'_{\widetilde{X}}\widetilde{Y}-\nabla_{\widetilde{X}}\widetilde{Y}%
&=\iota'_\xi\beta(\widetilde{X})\widetilde{Y}+\iota'_\xi\beta(\widetilde{Y})\widetilde{X}-\beta(\widetilde{X},\widetilde{Y})\xi\\*
&=\overline{\beta}(X)\widetilde{Y}+\overline{\beta}(Y)\widetilde{X}-\beta(\widetilde{X},\widetilde{Y})\xi.
\intertext{Similarly, we have}
\nabla'_{\widetilde{X}}\widetilde{Y}-\nabla_{\widetilde{X}}\widetilde{Y}%
&=\overline{\beta}(X)\widetilde{Y}+\overline{\beta}(Y)\widetilde{X}-\beta'(\widetilde{X},\widetilde{Y})\xi.
\end{align*}
Hence we have $\beta=\beta'$.
\end{proof}

\begin{remark}
Similar constructions for conformal structures without torsion are known, e.g.~\cite{BEG}.
We think that our framework also work for conformal structures with torsion.
\end{remark}

\section{examples}

We introduce examples of which the torsions are non-trivial and the curvatures are trivial.

Let $T^2=\R^2/\Z^2$ be the  standard torus and $(x^1,x^2)$ the standard coordinates.
We study projective structures of $T^2$ which are curvature-free and invariant under the standard $T^2$ action.
First of all, Christoffel symbols of connections are constants.

Let
\begin{align*}
\mathcal{T}&=\{\text{projective structures of $T^2$ invariant under the $T^2$-action and is curvature-free}\},\\*
\mathcal{T}'&=\{\tau\in\mathcal{T}\mid\text{$\tau$ is with torsion}\}.
\end{align*}
Let $\omega=(\omega^i,\omega^i{}_j,\omega_j)$ denote the normal projective connection associated with the projective structure given by an affine connection $\nabla$, $\sigma$ the section given by Proposition~\ref{prop2.9}.
Let $(\Omega^i,\Omega^i{}_j,\Omega_j)$ be the torsion and the curvature of $\omega$.
We have $\sigma^*\omega^i=dx^i$.
We have naturally $\widetilde{P}^2(T^2)\cong T^2\times\widetilde{G}^2$.
If $P\subset\widetilde{P}^2(T^2)$ is a projective structure, then we have $P\cong T^2\times H^2\subset T^2\times\widetilde{G}^2$.

\begin{example}
\label{ex6.1}
We consider an affine connection $\nabla$ of which the Christoffel symbols~are
\begin{alignat*}{6}
&\Gamma^1{}_{11}=1, &\quad & \Gamma^1{}_{12}=-\frac12, & \qquad & \Gamma^1{}_{21}=-\frac12, &\quad & \Gamma^1{}_{22}=0,\\*
&\Gamma^2{}_{11}=1, & & \Gamma^2{}_{12}=\frac32, & & \Gamma^2{}_{21}=-\frac12, & & \Gamma^2{}_{22}=-1.
\end{alignat*}
We set $g=(\delta^i{}_j,-\Gamma^i{}_{jk})\in\widetilde{G}^2$, which does not belong to $H^2$ because $\Gamma^2{}_{21}\neq\Gamma^2{}_{12}$.
We define $\sigma_0\colon T^2\to\widetilde{P}^2(T^2)$ by $\sigma(p)=(p,g)$ and define an $H^2$-subbundle $P$ of $\widetilde{P}^2(T^2)$ by
\[
P=\{u\in\widetilde{P}^2(T^2)\mid\exists p\in T^2,\ h\in H^2,\ u=\sigma_0(p).h\}.
\]
We have $\Gamma^\alpha{}_{\alpha1}=\frac12$, $\Gamma^\alpha{}_{\alpha2}=-\frac32$, $\Gamma^\alpha{}_{1\alpha}=\frac52$ and $\Gamma^\alpha{}_{2\alpha}=-\frac32$ so that
\begin{alignat*}{2}
&\mu_1=-1, & \quad & \mu_2=0,\\*
&\nu_1=-\frac12 & & \nu_2=\frac12.
\end{alignat*}
It follows that $\Pi^2{}_{11}=1$, $\Pi^2{}_{12}=1$, $\Pi^2{}_{21}=-1$ and $\Pi^i_{jk}=0$ for other triples $(i,j,k)$.
We have $\Pi_{11}=-1$ and $\Pi_{jk}=0$ for pairs $(j,k)$.
We have therefore that $\sigma^*\Omega^1=0$, $\sigma^*\Omega^2=-2dx^1\wedge dx^2$, $\sigma^*\Omega^i{}_j=0$ and $\sigma^*\Omega_j=0$.
Hence the connection $\nabla$ gives an element of $\mathcal{T}$ of which the torsion is non-trivial.

The normal TW-connection which corresponds to $\nabla$ is given as follows.
We have $\mathcal{E}(T^2)=T^2\times\R$.
Let $t$ be the standard coordinate for $\R$.
Then, the normal TW-connection is given by
\begin{align*}
\omega&=\begin{pmatrix}
\Pi^1{}_{1\alpha}dx^\alpha & \Pi^1{}_{2\alpha}dx^\alpha & 0\\
\Pi^2{}_{1\alpha}dx^\alpha & \Pi^2{}_{2\alpha}dx^\alpha & 0\\
-3\Pi_{1\alpha}dx^\alpha & -3\Pi_{2\alpha}dx^\alpha & 0
\end{pmatrix}-\frac13\begin{pmatrix}
dt & 0 & dx^1\\
0 & dt & dx^2\\
0 & 0 & dt
\end{pmatrix}\\*
&=\begin{pmatrix}
0 & 0 & 0\\
dx^1+dx^2 & -dx^1 & 0\\
3dx^1 & 0 & 0
\end{pmatrix}-\frac13\begin{pmatrix}
dt & 0 & dx^1\\
0 & dt & dx^2\\
0 & 0 & dt
\end{pmatrix},
\end{align*}
which is with torsion.
We have
\[
R(\omega)=\begin{pmatrix}
0 & 0 & 0\\
0 & 0 & \frac23dx^1\wedge dx^2\\
0 & 0 & 0
\end{pmatrix}
\]
so that $\omega$ is with torsion as a projective connection.
On the other hand, $\omega$ is curvature-free.
The correspondence between $(\Omega^i,\Omega^i{}_j,\Omega_j)$ and the components of $R(\omega)$ is given by Lemma~\ref{lem5.4}.
\end{example}

Projective structures with torsion are abundant even if we assume the curvatures to be trivial.

\begin{theorem}
The space $\mathcal{T}$ is a cubic subvariety of\/ $\R^6$ of dimension~$4$.
The space $\mathcal{T}'$ is an open subvariety of $\mathcal{T}$ and induces a subvariety of $\R P^5$ of dimension~$3$.
\end{theorem}
If we work in the complex category, then $\R^6$ and $\R P^5$ are replaced by $\C^6$ and~$\C P^5$.
\begin{proof}
We make use of notations in Lemma~\ref{lem2.21}.
Let $\psi^i{}_j=\Pi^i{}_{jk}dx^k$ and $\psi_j=\Pi_{jk}dx^k$.
We have
\[
\mu_j=\Pi^\alpha{}_{\alpha j}=-\Pi^\alpha{}_{j\alpha},
\]
where $\mu$ is the reduced torsion.
This is equivalent to
\begin{align*}
\stepcounter{equation}
\tag{\thetheorem-\theequation}
\label{eq6.2-0}
&2\Pi^1{}_{11}+\Pi^2{}_{21}+\Pi^2{}_{12}=0,\\*
\stepcounter{equation}
\tag{\thetheorem-\theequation}
\label{eq6.2-01}
&2\Pi^2{}_{22}+\Pi^1{}_{12}+\Pi^1{}_{21}=0.
\end{align*}
We have
\[
\Pi_{jk}=\frac13\left(2(\mu_\alpha\Pi^\alpha{}_{jk}+\Pi^\alpha{}_{j\beta}\Pi^\beta{}_{\alpha k})+(\mu_\alpha\Pi^\alpha{}_{kj}+\Pi^\alpha{}_{k\beta}\Pi^\beta{}_{\alpha j})\right).
\]
It follows that
\[
\Pi_{jk}=\frac13\left(2(-\Pi^\beta{}_{\alpha\beta}\Pi^\alpha{}_{jk}+\Pi^\alpha{}_{j\beta}\Pi^\beta{}_{\alpha k})+(-\Pi^\beta{}_{\alpha\beta}\Pi^\alpha{}_{kj}+\Pi^\alpha{}_{k\beta}\Pi^\beta{}_{\alpha j})\right)
\]
If $j=k$, then we have
\begin{align*}
-\Pi^\beta{}_{\alpha\beta}\Pi^\alpha{}_{11}+\Pi^\alpha{}_{1\beta}\Pi^\beta{}_{\alpha 1}
&=-\Pi^2{}_{12}\Pi^1{}_{11}-\Pi^2{}_{22}\Pi^2{}_{11}+\Pi^1{}_{12}\Pi^2{}_{11}+\Pi^2{}_{12}\Pi^2{}_{21},\\*
-\Pi^\beta{}_{\alpha\beta}\Pi^\alpha{}_{22}+\Pi^\alpha{}_{2\beta}\Pi^\beta{}_{\alpha 2}
&=-\Pi^1{}_{11}\Pi^1{}_{22}-\Pi^1{}_{21}\Pi^2{}_{22}+\Pi^1{}_{21}\Pi^1{}_{12}+\Pi^2{}_{21}\Pi^1{}_{22}.
\end{align*}
If $i\neq j$, then we have
\begin{align*}
-\Pi^\beta{}_{\alpha\beta}\Pi^\alpha{}_{12}+\Pi^\alpha{}_{1\beta}\Pi^\beta{}_{\alpha 2}
&=-\Pi^1{}_{21}\Pi^2{}_{12}+\Pi^2{}_{11}\Pi^1{}_{22}\\*
-\Pi^\beta{}_{\alpha\beta}\Pi^\alpha{}_{21}+\Pi^\alpha{}_{2\beta}\Pi^\beta{}_{\alpha 1}
&=-\Pi^2{}_{12}\Pi^1{}_{21}+\Pi^1{}_{22}\Pi^2{}_{11}.
\end{align*}
Hence we have
\begin{align}
\stepcounter{equation}
\tag{\thetheorem-\theequation}
\label{eq6.2-a}
\Pi_{11}&=-\Pi^2{}_{12}\Pi^1{}_{11}-\Pi^2{}_{22}\Pi^2{}_{11}+\Pi^1{}_{12}\Pi^2{}_{11}+\Pi^2{}_{12}\Pi^2{}_{21},\\*
\stepcounter{equation}
\tag{\thetheorem-\theequation}
\label{eq6.2-b}
\Pi_{12}&=\Pi_{21}=-\Pi^2{}_{12}\Pi^1{}_{21}+\Pi^1{}_{22}\Pi^2{}_{11},\\*
\stepcounter{equation}
\tag{\thetheorem-\theequation}
\label{eq6.2-c}
\Pi_{22}&=-\Pi^1{}_{11}\Pi^1{}_{22}-\Pi^1{}_{21}\Pi^2{}_{22}+\Pi^1{}_{21}\Pi^1{}_{12}+\Pi^2{}_{21}\Pi^1{}_{22}.
\end{align}
These are the defining equalities for $\Pi_{ij}$.

On the other hand, we have
\begin{alignat*}{3}
&\Omega^i&{}={}&\begin{pmatrix}
-\Pi^1{}_{12}+\Pi^1{}_{21}\\
-\Pi^2{}_{12}+\Pi^2{}_{21}
\end{pmatrix}dx^1\wedge dx^2,\\*
&\Omega^i{}_j&{}={}&\begin{pmatrix}
\Pi^1{}_{2k}\Pi^2{}_{1l}dx^k\wedge dx^l & (\Pi^1{}_{1k}\Pi^1{}_{2l}+\Pi^1{}_{2k}\Pi^2{}_{2l})dx^k\wedge dx^l\\
(\Pi^2{}_{1k}\Pi^1{}_{1l}+\Pi^2{}_{2k}\Pi^2{}_{1l})dx^k\wedge dx^l& \Pi^2{}_{1k}\Pi^1{}_{2l}dx^k\wedge dx^l
\end{pmatrix}\\*
& & &+\begin{pmatrix}
2\Pi_{12}-\Pi_{21} & \Pi_{22}\\
-\Pi_{11} & -2\Pi_{21}+\Pi_{12}
\end{pmatrix}dx^1\wedge dx^2,\\*
&\Omega_j&{}={}&\begin{pmatrix}
(\Pi_{1k}\Pi^1{}_{1l}+\Pi_{2k}\Pi^2{}_{1l})dx^k\wedge dx^l & (\Pi_{1k}\Pi^1{}_{2l}+\Pi_{2k}\Pi^2{}_{2l})dx^k\wedge dx^l
\end{pmatrix}.
\end{alignat*}
The projective structure is with torsion if and only if we have
\[
\stepcounter{equation}
\tag{\thetheorem-\theequation}
\label{eq6.2-d}
\Pi^1{}_{12}\neq\Pi^1{}_{21}\quad\text{or}\quad\Pi^2{}_{12}\neq\Pi^2{}_{21},
\]
while it is curvature-free, namely, $(\Omega^i{}_j,\Omega_j)=(0,0)$ if and only if we have
\begin{align}
\stepcounter{equation}
\tag{\thetheorem-\theequation}
\label{eq6.2-1}
&\Pi^1{}_{21}\Pi^2{}_{12}-\Pi^1{}_{22}\Pi^2{}_{11}+2\Pi_{12}-\Pi_{21}=0,\\*
\stepcounter{equation}
\tag{\thetheorem-\theequation}
\label{eq6.2-2}
&\Pi^1{}_{11}\Pi^1{}_{22}-\Pi^1{}_{12}\Pi^1{}_{21}+\Pi^1{}_{21}\Pi^2{}_{22}-\Pi^1{}_{22}\Pi^2{}_{21}+\Pi_{22}=0,\\*
\stepcounter{equation}
\tag{\thetheorem-\theequation}
\label{eq6.2-3}
&\Pi^2{}_{11}\Pi^1{}_{12}-\Pi^2{}_{12}\Pi^1{}_{11}+\Pi^2{}_{21}\Pi^2{}_{12}-\Pi^2{}_{22}\Pi^2{}_{11}-\Pi_{11}=0,\\*
\stepcounter{equation}
\tag{\thetheorem-\theequation}
\label{eq6.2-4}
&\Pi^2{}_{11}\Pi^1{}_{22}-\Pi^2{}_{12}\Pi^1{}_{21}-2\Pi_{21}+\Pi_{12}=0,\\*
\stepcounter{equation}
\tag{\thetheorem-\theequation}
\label{eq6.2-5}
&\Pi_{11}\Pi^1{}_{12}-\Pi_{12}\Pi^1{}_{11}+\Pi_{21}\Pi^2{}_{12}-\Pi_{22}\Pi^2{}_{11}=0,\\*
\stepcounter{equation}
\tag{\thetheorem-\theequation}
\label{eq6.2-6}
&\Pi_{11}\Pi^1{}_{22}-\Pi_{12}\Pi^1{}_{21}+\Pi_{21}\Pi^2{}_{22}-\Pi_{22}\Pi^2{}_{21}=0.
\end{align}
The equalities \eqref{eq6.2-2} and \eqref{eq6.2-3} are equivalent to the equalities~\eqref{eq6.2-c} and~\eqref{eq6.2-a}.
The equalities \eqref{eq6.2-1} and \eqref{eq6.2-4} are equivalent to the equality~\eqref{eq6.2-b}.
Hence we always have $\Omega^i{}_j=0$.

We consider $\tau=(\Pi^1{}_{12},\Pi^1{}_{21},\Pi^1{}_{22},\Pi^2{}_{11},\Pi^2{}_{12},\Pi^2{}_{21})$ as coordinates.
Let $F(\tau)$ be the left hand side of the equality~\eqref{eq6.2-5} and $G(\tau)$ be the left hand side of the equality~\eqref{eq6.2-6}.
We have
\begin{align*}
F(\tau)%
&=\frac12\Pi^2{}_{12}\Pi^2{}_{12}(\Pi^1{}_{12}-\Pi^1{}_{21})
+\Pi^2{}_{12}\Pi^2{}_{21}(\Pi^1{}_{12}-\Pi^1{}_{21})
-\Pi^2{}_{12}\Pi^1{}_{21}(\Pi^2{}_{12}-\Pi^2{}_{21})\\*
&\hphantom{{}={}}
+\frac12\Pi^2{}_{12}\Pi^2{}_{21}(\Pi^1{}_{12}-\Pi^1{}_{21})
+\frac12(\Pi^1{}_{12}\Pi^1{}_{12}-\Pi^1{}_{21}\Pi^1{}_{21})\Pi^2{}_{11}
+\Pi^1{}_{12}\Pi^2{}_{11}(\Pi^1{}_{12}-\Pi^1{}_{21})\\*
&\hphantom{{}={}}
+\Pi^1{}_{22}\Pi^2{}_{11}(\Pi^2{}_{12}-\Pi^2{}_{21}).
\end{align*}
Similarly, we have
\begin{align*}
G(\tau)%
&=\frac12\Pi^1{}_{21}\Pi^1{}_{21}(\Pi^2{}_{12}-\Pi^2{}_{21})
+\Pi^1{}_{21}\Pi^1{}_{12}(\Pi^2{}_{12}-\Pi^2{}_{21})
-\Pi^1{}_{21}\Pi^2{}_{12}(\Pi^1{}_{12}-\Pi^1{}_{21})\\*
&\hphantom{{}={}}
+\frac12\Pi^1{}_{21}\Pi^1{}_{12}(\Pi^2{}_{12}-\Pi^2{}_{21})
+\frac12(\Pi^2{}_{12}\Pi^2{}_{12}-\Pi^2{}_{21}\Pi^2{}_{21})\Pi^1{}_{22}
+\Pi^2{}_{21}\Pi^1{}_{22}(\Pi^2{}_{12}-\Pi^2{}_{21})\\*
&\hphantom{{}={}}
+\Pi^2{}_{11}\Pi^1{}_{22}(\Pi^1{}_{12}-\Pi^1{}_{21}).
\end{align*}

Suppose conversely that we can find $\tau=(\Pi^i{}_{jk})$, where $(i,j,k)\neq(1,1,1),(2,2,2)$, such that $F(\tau)=G(\tau)=0$.
We define $\Pi^1{}_{11}$ and $\Pi^2{}_{22}$ by~\eqref{eq6.2-0} and \eqref{eq6.2-01}, and $\Pi_{ij}$ by~\eqref{eq6.2-a}, \eqref{eq6.2-b} and \eqref{eq6.2-c}.
Then, the projective structure determined by $\tau$ is curvature-free.
It is with torsion if and only if the condition~\eqref{eq6.2-d} is satisfied.
Therefore, we have
\[
\mathcal{T}=\{\tau=(\Pi^i{}_{jk})\mid F(\tau)=G(\tau)=0\}.
\]
Note that if $\tau\in\mathcal{T}$ is torsion-free, then $\tau$ is flat, because $\tau$ is curvature-free.
In this example, if we assume $\Pi^1{}_{12}=\Pi^1{}_{21}$ and $\Pi^2{}_{12}=\Pi^2{}_{21}$, then $F(\tau)=G(\tau)=0$ are equal to zero so that $\Omega_j=0$.
This is analogous to the case of dimension greater than two.
In the latter case, the vanishing of $\Omega_j$ is guaranteed by Proposition~\ref{prop4.7}.

Affine connections which induce a given normal projective connection is obtained as follows.
Let $\nu_1,\nu_2\in\R$ be arbitrary, and set $\Gamma^i{}_{jk}=\Pi^i{}_{jk}-(\delta^i{}_j\nu_k+\delta^i{}_k\nu_j)$ for $(i,j,k)\neq(1,1,1),(2,2,2)$, where $\delta^i{}_j=\begin{cases}
1, & i=j,\\
0, & i\neq j
\end{cases}$.
We have then $-6\nu_1-(\Gamma^\alpha{}_{\alpha1}+\Gamma^\alpha{}_{1\alpha})=-6\nu_1-2\Gamma^1{}_{11}-\Pi^2{}_{21}-\Pi^2{}_{12}+\nu_1+\nu_1=-4\nu_1-2\Gamma^1{}_{11}+2\Pi^1{}_{11}$.
Hence we have $\Pi^1{}_{11}=\Gamma^1{}_{11}+2\nu_1$.
Similarly, we have $\Pi^2{}_{22}=\Gamma^2{}_{22}+2\nu_2$.
Thus defined affine connection induces the projective structure given by $\tau=(\Pi^i{}_{jk})$.

Finally, let $F^i{}_{jk}=\pdif{F}{\Pi^i{}_{jk}}$ and $G^i{}_{jk}=\pdif{G}{\Pi^i{}_{jk}}$.
We have
\begin{align*}
F^1{}_{12}(\tau)&=\frac12\Pi^2{}_{12}\Pi^2{}_{12}+\Pi^2{}_{12}\Pi^2{}_{21}+\frac12\Pi^2{}_{12}\Pi^2{}_{21}+\Pi^1{}_{12}\Pi^2{}_{11}\\*
&\hphantom{{}={}}%
+2\Pi^1{}_{12}\Pi^2{}_{11}-\Pi^2{}_{11}\Pi^1{}_{21},\\*
F^1{}_{21}(\tau)&=-\frac12\Pi^2{}_{12}\Pi^2{}_{12}-\Pi^2{}_{12}\Pi^2{}_{21}-\Pi^2{}_{12}(\Pi^2{}_{12}-\Pi^2{}_{21})-\frac12\Pi^2{}_{12}\Pi^2{}_{21}\\*
&\hphantom{{}={}}%
-\Pi^1{}_{21}\Pi^2{}_{11}-\Pi^1{}_{12}\Pi^2{}_{11},\\*
F^1{}_{22}(\tau)&=\Pi^2{}_{11}(\Pi^2{}_{12}-\Pi^2{}_{21}),\\*
F^2{}_{11}(\tau)&=\frac12(\Pi^1{}_{12}\Pi^1{}_{12}-\Pi^1{}_{21}\Pi^1{}_{21})+\Pi^1{}_{12}(\Pi^1{}_{12}-\Pi^1{}_{21})+\Pi^1{}_{22}(\Pi^2{}_{12}-\Pi^2{}_{21}),\\*
F^2{}_{12}(\tau)&=\Pi^2{}_{12}(\Pi^1{}_{12}-\Pi^1{}_{21})+\Pi^2{}_{21}(\Pi^1{}_{12}-\Pi^1{}_{21})-2\Pi^2{}_{12}\Pi^1{}_{21}+\Pi^1{}_{21}\Pi^2{}_{21}\\*
&\hphantom{{}={}}%
+\frac12\Pi^2{}_{21}(\Pi^1{}_{12}-\Pi^1{}_{21}),\\*
F^2{}_{21}(\tau)&=\Pi^2{}_{12}(\Pi^1{}_{12}-\Pi^1{}_{21})+\Pi^2{}_{12}\Pi^1{}_{21}+\frac12\Pi^2{}_{12}(\Pi^1{}_{12}-\Pi^1{}_{21})-\Pi^1{}_{22}\Pi^2{}_{11},\\
G^1{}_{12}(\tau)&=-\Pi^1{}_{21}(\Pi^2{}_{21}-\Pi^2{}_{12})-\Pi^1{}_{21}\Pi^2{}_{12}-\frac12\Pi^1{}_{21}(\Pi^2{}_{21}-\Pi^2{}_{12})+\Pi^2{}_{11}\Pi^1{}_{22},\\*
G^1{}_{21}(\tau)&=-\Pi^1{}_{21}(\Pi^2{}_{21}-\Pi^2{}_{12})-\Pi^1{}_{12}(\Pi^2{}_{21}-\Pi^2{}_{12})+2\Pi^1{}_{21}\Pi^2{}_{12}-\Pi^2{}_{12}\Pi^1{}_{12}\\*
&\hphantom{{}={}}%
-\frac12\Pi^1{}_{12}(\Pi^2{}_{21}-\Pi^2{}_{12}),\\*
G^1{}_{22}(\tau)&=-\frac12(\Pi^2{}_{21}\Pi^2{}_{21}-\Pi^2{}_{12}\Pi^2{}_{12})-\Pi^2{}_{21}(\Pi^2{}_{21}-\Pi^2{}_{12})-\Pi^2{}_{11}(\Pi^1{}_{21}-\Pi^1{}_{12}),\\*
G^2{}_{11}(\tau)&=-\Pi^1{}_{22}(\Pi^1{}_{21}-\Pi^1{}_{12}),\\*
G^2{}_{12}(\tau)&=\frac12\Pi^1{}_{21}\Pi^1{}_{21}-\Pi^1{}_{21}\Pi^1{}_{12}+\Pi^1{}_{21}(\Pi^1{}_{21}-\Pi^1{}_{12})+\frac12\Pi^1{}_{21}\Pi^1{}_{12}\\*
&\hphantom{{}={}}%
+\Pi^2{}_{12}\Pi^1{}_{22}+\Pi^2{}_{21}\Pi^1{}_{22},\\*
G^2{}_{21}(\tau)&=-\frac12\Pi^1{}_{21}\Pi^1{}_{21}-\Pi^1{}_{21}\Pi^1{}_{12}-\frac12\Pi^1{}_{21}\Pi^1{}_{12}-\Pi^2{}_{21}\Pi^1{}_{22}\\*
&\hphantom{{}={}}%
-2\Pi^2{}_{21}\Pi^1{}_{22}+\Pi^1{}_{22}\Pi^2{}_{12}.
\end{align*}
If $\Pi^1{}_{12}=\Pi^1{}_{21}$ and if $\Pi^2{}_{12}=\Pi^2{}_{21}$, then we have
\begin{alignat*}{3}
F^1{}_{12}(\tau)&=2(\Pi^2{}_{12}\Pi^2{}_{12}+\Pi^1{}_{12}\Pi^2{}_{11}), & \quad & G^1{}_{12}(\tau)=-\Pi^1{}_{12}\Pi^2{}_{12}+\Pi^1{}_{22}\Pi^2{}_{11},\\*
F^1{}_{21}(\tau)&=-2(\Pi^2{}_{12}\Pi^2{}_{12}+\Pi^1{}_{12}\Pi^2{}_{11}), & \quad & G^1{}_{21}(\tau)=\Pi^1{}_{12}\Pi^2{}_{12},\\*
F^1{}_{22}(\tau)&=F^2{}_{11}(\tau)=0, & \quad & G^1{}_{22}(\tau)=G^2{}_{11}(\tau)=0,\\*
F^2{}_{12}(\tau)&=-\Pi^1{}_{12}\Pi^2{}_{12}, & \quad & G^2{}_{12}(\tau)=2(\Pi^1{}_{12}\Pi^1{}_{12}+\Pi^2{}_{12}\Pi^1{}_{22}),\\*
F^2{}_{21}(\tau)&=\Pi^1{}_{12}\Pi^2{}_{12}-\Pi^1{}_{22}\Pi^2{}_{11}, & \quad & G^2{}_{21}(\tau)=-2(\Pi^1{}_{12}\Pi^1{}_{12}+\Pi^2{}_{12}\Pi^1{}_{22}).
\end{alignat*}
Hence $\begin{pmatrix}
\pdif{F}{\tau}(\tau) & \pdif{G}{\tau}(\tau)
\end{pmatrix}$ is of rank $2$ for almost every $\tau$.
If $\tau\in\mathcal{T}'$, then we have $\Pi^1{}_{12}\neq\Pi^1{}_{21}$ or $\Pi^2{}_{12}\neq\Pi^2{}_{21}$.
In particular, one of $\Pi^1{}_{12},\Pi^1{}_{21},\Pi^2{}_{12}$ and $\Pi^2{}_{21}$ is non-zero.
Hence $\mathcal{T}'$ induces an open subvariety of $\R P^5$.
\end{proof}

An open subset of dimension $4$ of\/ $\mathcal{T}$ exists by the implicit function theorem, however, it seems difficult to find explicit ones.
We will present a family of elements of $\mathcal{T}$ with three parameters.

\begin{example}
Suppose that $\Pi^1{}_{12}=\Pi^1{}_{21}=\Pi^1{}_{22}=0$.
Then we have $F(\tau)=G(\tau)=0$ and $\Pi^2{}_{22}=0$.
It follows that
\begin{align*}
\Pi_{11}&=-\Pi^2{}_{12}\Pi^1{}_{11}+\Pi^2{}_{21}\Pi^2{}_{12}\\*
&=\frac32\Pi^2{}_{12}\Pi^2{}_{21}+\frac12\Pi^2{}_{12}\Pi^2{}_{12},\\*
\Pi_{12}&=\Pi_{21}=0,\\*
\Pi_{22}&=0.
\end{align*}
Let $a=\Pi^2{}_{11}$, $b=\Pi^2{}_{12}$ and $c=\Pi^2{}_{21}$.
The normal TW-connection is given by
\[
\omega=\begin{pmatrix}
-\frac{b+c}2dx^1 & 0 & 0\\
adx^1+bdx^2 & cdx^1 & 0\\
-\frac32(3bc+b^2)dx^1 & 0 &0
\end{pmatrix}-\frac13\begin{pmatrix}
dt & & dx^1\\
& dt & dx^2\\
& & dt
\end{pmatrix}
\]
We have $\Omega^1=0$ and $\Omega^2=\frac13(b-c)dx^1\wedge dx^2$.
The torsion of $\omega$ is equal to $\begin{pmatrix}
0\\
-b+c
\end{pmatrix}dx^1\wedge dx^2$.
By setting $a=b=1$ and $c=-1$, we obtain Example~\ref{ex6.1}.
Note that the ratio $a:b:c$ is relevant.
\end{example}

We have another kind of a one-parameter family.

\begin{example}
Let $\Pi^1{}_{12}=-\Pi^1{}_{21}=\sin\theta$ and $\Pi^2{}_{21}=-\Pi^2{}_{12}=\cos\theta$.
We have $\Pi^1{}_{11}=\Pi^2{}_{22}=0$ by~\eqref{eq6.2-0} and \eqref{eq6.2-01}.
On the other hand, we have
\begin{align*}
F(\tau)&=2(\sin^2\theta-(\cos\theta)\Pi^1{}_{22})\Pi^2{}_{11},\\*
G(\tau)&=-2(\cos^2\theta-(\sin\theta)\Pi^2{}_{11})\Pi^1{}_{22}.
\end{align*}
\begin{enumerate}
\item
If $\sin\theta=0$, then we have $\cos\theta\neq0$.
Since $G(\tau)=0$, we have $\Pi^1{}_{22}=0$.
Hence $\Pi_{12}=\Pi_{21}=0$ by~\eqref{eq6.2-b}.
We have $\Pi_{11}=-1$ and $\Pi_{22}=0$ by~\eqref{eq6.2-a} and~\eqref{eq6.2-c}.
The normal TW-connection is given by
\[
\begin{pmatrix}
0 & 0 & 0\\
\Pi^2{}_{11}dx^1\pm dx^2 & \mp dx^1 & 0\\
3dx^1 & 0 & 0
\end{pmatrix}-\frac13\begin{pmatrix}
dt & & dx^1\\
& dt & dx^2\\
& & dt
\end{pmatrix},
\]
where the double signs correspond and $\Pi^2_{11}$ is arbitrary.
\item
If $\cos\theta=0$, then the normal TW-connection is given by
\[
\begin{pmatrix}
\pm dx^2 & \mp dx^1+\Pi^1{}_{22}dx^2 & 0\\
0 & 0 & 0\\
0 & 3dx^2 & 0
\end{pmatrix}-\frac13\begin{pmatrix}
dt & & dx^1\\
& dt & dx^2\\
& & dt
\end{pmatrix}.
\]
\item
If $\sin\theta\neq0$ and if $\cos\theta\neq0$, then either $\Pi^1{}_{22}=\Pi^2{}_{11}=0$ or $\Pi^1{}_{22}=\frac{\sin^2\theta}{\cos\theta}$, $\Pi^2{}_{11}=\frac{\cos^2\theta}{\sin\theta}$.
In the first case, the normal TW-connection is given by
\[
\begin{pmatrix}
\sin\theta dx^2 & -\sin\theta dx^1 & 0\\*
-\cos\theta dx^2 & \cos\theta dx^1 & 0\\
0 & 0 & 0
\end{pmatrix}-\frac13\begin{pmatrix}
dt & & dx^1\\
& dt & dx^2\\
& & dt
\end{pmatrix}.
\]
In the second case, the normal TW-connections is given by
\[
\begin{pmatrix}
\sin\theta dx^2 & -\sin\theta dx^1+\frac{\sin^2\theta}{\cos\theta}dx^2 & 0\\
\frac{\cos^2\theta}{\sin\theta}dx^1-\cos\theta dx^2 & \cos\theta dx^1 & 0\\
0 & 0 & 0
\end{pmatrix}-\frac13\begin{pmatrix}
dt & & dx^1\\
& dt & dx^2\\
& & dt
\end{pmatrix}.
\]
\end{enumerate}
In the both cases, the torsion is given by $2\begin{pmatrix}
-\sin\theta\\
\hphantom{-}\cos\theta
\end{pmatrix} dx^1\wedge dx^2$.
Hence the ratio $K^1{}_{12}:K^2{}_{12}$ can take any value.
The latter connection can be slightly generalized~as
\[
\begin{pmatrix}
r\sin^2\theta\cos\theta dx^2 & -r(\sin^2\theta\cos\theta dx^1+\sin^3\theta dx^2) & 0\\
r(\cos^3\theta dx^1-\sin\theta\cos^2\theta dx^2) & r\sin\theta\cos^2\theta dx^1 & 0\\
0 & 0 & 0
\end{pmatrix}-\frac13\begin{pmatrix}
dt & & dx^1\\
& dt & dx^2\\
& & dt
\end{pmatrix}.
\]
of which the torsion is given by $2r\sin\theta\cos\theta\begin{pmatrix}
-\sin\theta\\
\hphantom{-}\cos\theta
\end{pmatrix}dx^1\wedge dx^2$.
\end{example}

\end{document}